\documentclass[10pt,reqno]{amsart}

\pdfoutput = 1

\usepackage[english]{babel}
\usepackage{amsmath,amssymb,amsthm}
\usepackage{amsfonts}
\usepackage{amsxtra}

\usepackage{array,dcolumn}
\usepackage{graphicx}

\usepackage[T1]{fontenc}
\usepackage[utf8]{inputenc}
\usepackage[activate={true,nocompatibility},spacing,kerning]{microtype}
\usepackage[pdftex,
        colorlinks,
        hyperfootnotes=false,
        pdffitwindow=true,
        plainpages=false,
        pdfpagelabels=true,
        pdfpagemode=UseOutlines,
        pdfpagelayout=SinglePage,
        pdfpagescrop={54.7664 45.4919 562.83 738.076},  
        pdftitle={Joint Distribution},
        pdfauthor={Folkmar Bornemann, Technische Universität München},
        pagebackref,
        hyperindex,
        filecolor=purple,
        urlcolor=purple,
        citecolor=blue,
        linkcolor=blue]{hyperref}
\usepackage{bm}
\usepackage[sc,osf]{mathpazo}
\microtypecontext{spacing=nonfrench}

\newcommand{\mathsym}[1]{{}}
\newcommand{\unicode}[1]{{}}

\theoremstyle{plain}

\newtheorem{lemma}{Lemma}

\newtheorem{proposition}{Proposition}

\theoremstyle{definition}

\theoremstyle{remark}
\newtheorem{remark}{Remark}

\newcommand{\Z}{\mathbb Z}
\newcommand{\R}{\mathbb R}
\newcommand{\C}{\mathbb C}

\newcommand{\half}{\tfrac{1}{2}}
\newcommand{\quarter}{\tfrac{1}{4}}
\newcommand{\ssqrt}[2][]{ \hspace{-0.5pt} \sqrt[#1]{#2}}
\renewcommand{\leq}{\leqslant}
\renewcommand{\geq}{\geqslant}
\newcommand{\E}{\mathbb{E}}

\begin{document}


\title[Joint distribution of \dots at the soft edge of unitary ensembles]{Joint distribution of the first and second eigenvalues at the soft edge of unitary ensembles}
\author{N. S. Witte}
\address{Dept. of Mathematics and Statistics, University of Melbourne, Victoria 3010, Australia}
\email{n.witte@ms.unimelb.edu.au}
\author{F. Bornemann}
\address{Zentrum Mathematik -- M3, Technische Universität München,
      80290~München, Germany}
\email{bornemann@tum.de}
\author{P. J. Forrester}
\address{Dept. of Mathematics and Statistics, University of Melbourne, Victoria 3010, Australia}
\email{p.forrester@ms.unimelb.edu.au}
\date{\today}


\begin{abstract}
The density function for the joint distribution of the first and second eigenvalues at the soft edge of 
unitary ensembles is found in terms of a Painlev\'e II transcendent and its associated isomonodromic
system. As a corollary, the density function for the spacing between these two eigenvalues is similarly
characterized.The particular solution of Painlev\'e II that arises is a double shifted B\"acklund
transformation of the Hasting-McLeod solution, which applies in the case of the distribution of 
the largest eigenvalue at the soft edge. Our deductions are made by employing the hard-to-soft edge
transitions to existing results for the joint distribution of the first and second eigenvalue 
at the hard edge \cite{FW_2007}. In addition recursions under  $a \mapsto a+1$ of quantities specifying 
the latter are obtained. A Fredholm determinant type characterisation is used to provide accurate numerics
for the distribution of the spacing between the two largest eigenvalues.
\end{abstract} 


\maketitle


\section{Introduction}\label{newIntro}
\setcounter{equation}{0}

Fundamental to random matrix theory and its applications is the soft edge scaling limit of unitary invariant ensembles.
As a concrete example, consider the Gaussian unitary ensemble, specified by the measure on complex Hermitian matrices
$H $ proportional to $\exp(-\mbox{Tr}\,H^2)({\rm d} { H})$. This measure is unchanged by the mapping 
${ H} \mapsto U  H U^{\dagger}$, for $U$ unitary, and is thus a unitary invariant. To leading order the support 
of the spectrum is $(-\ssqrt{2N},\ssqrt{2N})$, although there is a nonzero probability of eigenvalues in 
$(-\infty, -\ssqrt{2N}) \cup (\ssqrt{2N},\infty)$, and for this reason the neighbourhood of $\ssqrt{2N}$ (or $-\ssqrt{2N}$)
is referred to as the soft edge. Moreover, upon the scaling of the eigenvalues $\lambda_{\ell} \mapsto \ssqrt{2N} + X_{\ell}/\ssqrt{2}N^{1/6}$, 
the mean spacing between eigenvalues in the neighbourhood of the largest eigenvalue is of order unity. Taking the $N \rightarrow \infty$ limit 
with this scaling gives a well-defined statistical mechanical state, which is an example of a determinantal point process, and defined in terms of its
$k$-point correlation functions by 
\begin{align}
\rho_{(k)}^{\rm soft}(x_1,\ldots,x_k) = \det \left[ K^{\rm soft} (x_j,x_{\ell}) \right]_{j,\ell = 1,\ldots,k} \ ,  
\label{eq:2f}
\end{align}
where $K^{\rm soft}$ -- referred to as the correlation kernel -- is given in terms of Airy functions by
\begin{align}
K^{\rm soft}(x,y):= \frac{ {\rm Ai}(x) {\rm Ai}^{\prime}(y) - {\rm Ai}(y) {\rm Ai}^{\prime}(x)   } {x-y} .    
\label{eq:2fa}
\end{align}

The determinantal form \eqref{eq:2f} implies that in the soft edge scaled state, the probability of there being 
no eigenvalues in the interval $(s,\infty)$, is given by \cite{Fo_1993}
\begin{align}
 {\rm E}_2^{\rm soft}(0;(s,\infty)) 
       & =   1 + \sum_{k=1}^{\infty} \frac{ (-1)^k}{k!} \int_{s}^{\infty} {\rm d}x_1 \ldots \int_{s}^{\infty} {\rm d}x_k  \ \rho_{(k)}^{\rm soft}(x_1,\ldots,x_k) , \notag \\
       & =    \det \left(1 - {\mathbb K}^{\rm soft}_{(s,\infty)}\right),													 
\label{eq:2g}
\end{align}
where $ {\mathbb K}^{\rm soft}_{(s,\infty)} $ is the integral operator on $(s,\infty)$ with kernel $K^{\rm soft}(x,y)$ (as given in \eqref{eq:2fa}). The first 
equality in \eqref{eq:2g} is generally true for a one-dimensional point process, while the second equality follows from the Fredholm theory 
\cite{WW_1958}.
The structure of the kernel \eqref{eq:2fa} makes it of a class referred to as integrable \cite{IIKS_1990}, and generally this class of integrable 
kernels have intimate connections to integrable systems. Indeed one has that \cite{TW_1994a}
\begin{align}
 \det \left(1 - {\mathbb K}^{\rm soft}_{(s,\infty)}\right) = 
 	\exp\left( - \int_s^{\infty}(t-s) q^2(t) {\rm d}t \right),	 
\label{eq:3.1} 
 \end{align}
where $q(t) $ satisfies the particular Painlev\'e II ordinary differential equation ($\;\dot{}\equiv d/dt$)
\begin{align}
 \ddot{q} & = 2q^3+tq ,
\label{eq:4.1} \\
\noalign{\hbox{subject to the boundary condition}}
q(t) & \mathop{\sim}_{t \rightarrow \infty} {\rm Ai}(t).				 
\label{eq:4.1a}
\end{align}

Our interest in this paper is in the joint distribution of the largest and second largest eigenvalue at the soft edge, and the corresponding 
distribution of the spacing between them. Let $p_{(2)}^{\rm soft}(x_1,x_2)$, $x_1>x_2$, denote the density function of the joint distribution. 
Then analogous to the first equality in \eqref{eq:2g} we have

\begin{multline}
 p_{(2)}^{\rm soft}(x_1,x_2)  = \det 	\begin{bmatrix} K^{\rm soft}(x_1,x_1) &  K^{\rm soft}(x_1,x_2) \\
					  K^{\rm soft}(x_2,x_1) & K^{\rm soft}(x_2,x_2)  \end{bmatrix}  
\\
 + \sum_{k=1}^{\infty} \frac{ (-1)^k}{k!} \int_{x_2}^{\infty} {\rm d}y_1 \cdots \int_{x_2}^{\infty} {\rm d}y_k  
\\ \times
   \det \begin{bmatrix} K^{\rm soft}(x_1,x_1) &  K^{\rm soft}(x_1,x_2)  & \left[ K^{\rm soft}(x_1,y_{\ell}) \right]_{\ell=1,\ldots,k}  \\
		K^{\rm soft}(x_2,x_1) &  K^{\rm soft}(x_2,x_2)  & \left[ K^{\rm soft}(x_2,y_{\ell}) \right]_{\ell=1,\ldots,k}  \\
		\left[ K^{\rm soft}(y_j,x_1) \right]_{j=1,\ldots,k}  & \left[ K^{\rm soft}(y_j,x_2) \right]_{j=1,\ldots,k} & \left[ K^{\rm soft}(y_j,y_{\ell}) \right]_{j,\ell=1,\ldots,k}
	\end{bmatrix} .   
\label{eq:c1} 
\end{multline}
With  $A^{\rm soft}(s)$ denoting the density function for the spacing between the two largest eigenvalues we have 
\begin{equation}
A^{\rm soft}(s) = \int_{- \infty}^{\infty} {\rm d}x \, p^{\rm soft}_{(2)}(x+s,x). 		 \label{eq:c2}
\end{equation}   
					  
We seek to characterize \eqref{eq:c1} and \eqref{eq:c2} in a form analogous to \eqref{eq:3.1}. This involves functions which are components of a solution 
of a particular isomonodromic problem relating to the PII equation. Such characterizations have appeared in other problems in random matrix 
theory and related growth processes \cite{BR_2000}, \cite{FW_2007}, \cite{CIK_2010}, \cite{S_2012}. 

The starting point for us is our earlier study \cite{FW_2007} specifying the joint distribution of the first and second smallest eigenvalues, and 
the corresponding spacing distribution between these eigenvalues, at the hard edge of unitary ensembles. In random matrix theory the latter 
applies when the eigenvalue density is strictly zero on one side of its support, and is specified by the determinantal point process with correlation
kernel
\begin{align}
K^{{\rm hard,} \, a}(x,y) =   \frac{\ssqrt{y}J_a(\ssqrt{x})J'_a(\ssqrt{y})-\ssqrt{x}J'_a(\ssqrt{x})J_a(\ssqrt{y})}{2(x-y)} ,    
\label{eq:1.31}
\end{align} 
where $x,y>0$. Note the dependence on the parameter $a$ ($a>-1$) which physically represents a repulsion from the origin. The relevance to 
the study of the soft edge is that upon the scaling 
\begin{align}
x \mapsto a^2 [ 1-2^{2/3}a^{-2/3}x ] ,
\label{eq:1.31a}
\end{align}
(and similarly $y$), as $a \rightarrow \infty$ the hard edge kernel \eqref{eq:1.31} limits to the soft edge kernel, and consequently the 
hard edge state as defined by its correlation functions limits to the soft edge state \cite{BF_2003}. Thus our task is to compute this limit in the 
expressions from \cite{FW_2007}. Moreover, recurrences under the mapping of the latter will be specified.

In Section~\ref{S:HEdbn} the evaluation of the joint distribution of the first and second eigenvalue at the hard edge from \cite{FW_2007} is revised.
This involves quantities relating to the Hamiltonian formulation of the Painlev\'e III$^{\prime}$ equation, and to an isomonodromic problem for 
the generic Painlev\'e III$^{\prime}$ equation. Details of these aspects are discussed in separate subsections, with special emphasis placed on the
transformation of the relevant quantities under the mapping $a \mapsto a +1$. Second order recurrences are obtained. In Subsection \ref{SS:integer_a}
initial conditions for these recurrences are specified. Section~\ref{S:H2S} is devoted to the computation of the hard-to-soft edge scaling of the quantities 
occurring in the evaluation of the joint distribution of the first and second eigenvalue at the hard edge. This allows us to evaluate the joint 
distribution of the first and second eigenvalues at the soft edge in terms of a Painlev\'e II transcendent and its associated isomonodromic system.

\section{\sc Hard edge $a>0$ joint distribution of the first and second eigenvalues}\label{S:HEdbn}
\setcounter{equation}{0}
\subsection{The result from {\bf \cite{FW_2007}}}
Let $p_{(2)}^{{\rm hard,} \, a} (x_1,x_2)$, $x_2>x_1$ denote the joint distribution of the smallest and second smallest eigenvalues at the 
hard edge with unitary symmetry. It was derived in \cite{FW_2007} that
\begin{multline}
 p_{(2)}^{{\rm hard,} \, a}(s-z,s)
\\
			= \frac{z^2 s^a (s-z)^a e^{-s/4}} {4^{2a+3} \Gamma(a+1) \Gamma(a+2) \Gamma^2(a+3) }
				\exp\left( \int_0^s \frac{{\rm d}r}{r}  \left[ \nu(r) + 2C(r) \right] \right) \left(u \partial_z v - v \partial_z u \right) .
\label{eq:1.16}
\end{multline}
Here $\nu(s)$ is the solution of the second-order, second-degree ODE (${}^\prime \equiv {\rm d}/{\rm d}s$) -- a variant of the $\sigma$-form of the third
Painlev\'e equation, \cite[Eq.~(5.25)]{FW_2007}
\begin{equation}
 s^2(\nu'')^2-(a+2)^2(\nu')^2+\nu'(4\nu'-1)(s \nu'-\nu)+\tfrac{1}{2}a(a+2)\nu'-\tfrac{1}{16}a^2 = 0 .   
\label{PIIIsigma}
\end{equation} 
Important to our subsequent workings is the fact that $p_{(1)}^{{\rm hard,} \, a}(s)$ -- the probability density function for the smallest eigenvalue
at the hard edge of an ensemble with unitary symmetry -- can be expressed in terms of $\nu(s)$ by \cite{FW_2002a}, \cite[Eq.~(8.93)]{rmt_Fo}
\begin{align}
p_{(1)}^{{\rm hard,} \, a}(s)  = \frac{s^a}{ 2^{2a+2} \Gamma(a+1)\Gamma(a+2) } \exp \left( \int_0^s \left( \nu(t) - \frac{t}{4} \right) 
		 \frac{ {\rm d}t} {t} \right).
\label{eq:x1}
\end{align}
 
To define $C(s)$, introduce the auxiliary quantity $\mu = \mu(s)$ according to \cite[Eq.~(5.25)]{FW_2007},
\begin{align}
 \mu + s = 4s \nu^{\prime}.
\label{HE_DEnu}
\end{align}
Then, according to \cite[Eq.~(5.20)]{FW_2007}, $C$ is specified by 
\begin{align}
   2C+a+3 = s\frac{{\mu'}-2}{\mu} .    
\label{HE_Crel}
\end{align}
These quantities are closely related to the Hamiltonian variables of Okamoto's theory for PIII$^{\prime}$, as will be seen
subsequently.
 
The variables $u(z;s)$ and $v(z;s)$ are the components of a solution to the associated isomonodromic problem for the
generic third Painlev\'e equation or the degenerate fifth Painlev\'e equation. They satisfy the Lax pair \cite[Eqs~(5.34-7)]{FW_2007},
on that domain $s>z$, $s,z \in \mathbb{R}$, with real $a>-1$, $a\in \mathbb{R}$,
\begin{align}
 z(s-z)\partial_{z}u & = -C z u-(\mu +z)v,
\label{HE_S:a}\\
 z(s-z)\partial_{z}v & = -z\left[\xi +\tfrac{1}{4}(z-s)\right] u+[-2s+(C+a+2)z]v ,
\label{HE_S:b} 
\end{align}
and
\begin{align}
  (s-z)s\partial_su & =  zCu+(\mu+s)v ,
\label{HE_D:a} \\
  (s-z)s\partial_sv & = z\xi u-[s(2C+a)-zC]v ,
\label{HE_D:b}
\end{align}
where $\xi$ is a further auxiliary quantity specified by (\cite[Eq.~(5.19)]{FW_2007})
\begin{align}
   \xi = -\frac{sC(C+a)}{\mu+s}.        
\label{HE_xidefn}    
\end{align}

For \eqref{HE_S:a}--\eqref{HE_D:b} to specify a unique solution appropriate boundary conditions must 
be specified. Their explicit form can be found in \cite{FW_2007}.
 
\subsection{Okamoto PIII$^{\prime}$ theory} 
 We seek to make the links to the Hamiltonian theory of the third Painlev\'e equation in order to draw upon the results of Okamoto \cite{Ok_1987c}, 
\cite{Ok_1987b} and the work by Forrester and Witte \cite{FW_2002a}. As given in these works the Hamiltonian theory of Painlev\'e III' can formulated in the 
variables $ \{q,p;s,H\} $ where the Hamiltonian itself is given by ( $ ' \equiv d/ds $)
\begin{equation}
  sH = q^2 p^2 - (q^2 + v_1q - s)p + \tfrac{1}{2}(v_1 + v_2) q.
\label{HIII}
\end{equation} 
With $H$ so specified the corresponding Hamilton equations of motion are
\begin{align}
  sq' & = 2q^2p - (q^2 + v_1 q - s) ,
\label{EoMIII:a}\\
  sp' & = -2qp^2+(2q+v_1)p-\tfrac{1}{2}(v_1+v_2) .
\label{EoMIII:b}
\end{align}

From these works its known that the canonical variables can be found from the time evolution of the
Hamiltonian itself by
\begin{align}
  p & = h' + \tfrac{1}{2} ,
\label{HamIII:a}\\
  q & = \frac{sh'' - v_1 h' + \frac{1}{2} v_2 }{\frac{1}{2}(1 - 4 (h')^2)} ,
\label{HamIII:b}
\end{align}
where
\begin{equation}\label{H4}
  h = sH + \tfrac{1}{4} v_1^2 - \tfrac{1}{2} s.
\end{equation}
In turn the Painlev\'e III' $\sigma$-function is related to the Hamiltonian by
\begin{equation}
   \sigma_{\rm III'}(s) := - (sH) \big|_{s \mapsto s/4} - \tfrac{1}{4}v_1(v_1 - v_2) + \tfrac{1}{4}s .
\end{equation} 
 
 In the work \cite{FW_2007} (see Prop. 5.21) the identification made with the Painlev\'e III' system gave the parameter
correspondence $ v_1 = a+2$, $v_2 = a-2 $ and  
\begin{equation}
   \nu(s) = -\sigma_{\rm III'}(s)+\tfrac{1}{4}s-a-2 .
\label{HE_sigma-nu}
\end{equation} 
The quantity $C$ appearing in \eqref{eq:1.16} and the auxiliary quantities  $\mu$ and $\xi $ can be related to $p$ and $q$ in the corresponding
 Hamiltonian system. 
 
\begin{proposition}  
The variables $ \mu, C, \xi $ are related to the canonical Painlev\'e III' co-ordinates by
\begin{align}
  \frac{\mu}{s} & = (p-1)\big|_{s \mapsto s/4} ,
\label{III_mu}\\
   C & = -qp\big|_{s \mapsto s/4} ,
\label{III_C}\\
  \xi & = q(a-qp)\big|_{s \mapsto s/4} .
\label{III_xi}
\end{align}
\end{proposition}
\begin{proof}
From Eq. (5.21) of \cite{FW_2007} and \eqref{HE_sigma-nu} we compute that
\begin{equation}
  \nu(s) = h(s/4) - \tfrac{1}{4}(a+2)^2 + \tfrac{1}{8}s .
\end{equation} 
Differentiating this and employing the relations \eqref{HE_DEnu} and \eqref{HamIII:a} we find \eqref{III_mu}.
Using \eqref{HE_Crel} we note that $ 4s^2{\nu''} = 2s+(2C+a+2)\mu $ and with the above equation and
\eqref{HamIII:b} we deduce \eqref{III_C}. Equation~\eqref{III_xi} then follows from \eqref{HE_xidefn}.
\end{proof}

For the Hamiltonian (\ref{HIII}), Okamoto \cite{Ok_1987c} has identified two Schlesinger
transformations with the property
\begin{equation}\label{T}
  T_1(v_1, v_2) = (v_1+1, v_2+1),  \qquad
  T_2(v_1, v_2) = (v_1+1, v_2 - 1),
\end{equation}
and has furthermore specified the corresponding mapping of $p$ and $q$. Recalling $(v_1,v_2)$ in terms of $a$ above 
\eqref{HE_sigma-nu}, we see that in the present case $T_1$ corresponds to $a \mapsto a + 1$. Reading from \cite{FW_2002a}
Eqs. (4.40-3) gives the following result.

\begin{proposition}[{\cite[Eqs. (4.40-3)]{FW_2002a}}]\label{recurHam}    
The Painlev\'e III' canonical variables $ q[a](s), p[a](s) $ satisfy coupled recurrence relations in $ a $
\begin{align}
 q[a+1] & = -\frac{s}{q[a]}+\frac{(a+1)s}{q[a]\left(q[a]\left(p[a]-1\right)-2\right)+s} ,
\\
 p[a+1] & = \frac{1}{s}q[a]\left(q[a]\left(p[a]-1\right)-2\right)+1 .
\end{align}
The reader should note that we haven't made the scale change $ s \mapsto s/4 $ here.
The initial conditions are given by \eqref{H_a=0} below for the sequence $ a \in \Z_{\geq 0} $.
\end{proposition}

\subsection{Isomonodromic system}\label{SS:HE_isomonodromy}
We now turn our attention to the isomonodromic system \eqref{HE_S:a} -\eqref{HE_D:b} for $u,v$ associated with the Painlev\'e system. 
Following the development of \cite{FW_2007} we define the matrix variable
\begin{equation}
   \Psi(z;s) = \begin{pmatrix} u(z;s) \\ v(z;s) \end{pmatrix} .
\label{HE_Psi}
\end{equation}

To begin with our interest is in the recurrence relations that are satisfied by $u$ and $v$ upon the 
mapping $a \mapsto a+1$.
\begin{proposition}\label{recurLax}
The isomonodromic components $ u, v $ satisfy linear coupled recurrence relations in $ a $
\begin{align}
 u[a+1] & = \frac{s}{s-z}\left(u[a]-\frac{C[a]+a}{\xi[a]}v[a]\right) ,
\label{HE_uST}
\\
 v[a+1] & = -\frac{s}{s-z}\frac{1}{4}\frac{C[a]+a}{\xi[a]}\left(z u[a]-s\frac{C[a]+a}{\xi[a]}v[a]\right) .
\label{HE_vST}
\end{align}
The initial conditions are given by \eqref{uv_a=0} for the sequence $ a \in \Z_{\geq 0} $.
\end{proposition}

\begin{proof}
The result \eqref{eq:1.16} from \cite{FW_2007} was derived as the hard edge scaling limit of the joint distribution of the first
and second eigenvalues in the finite $ N $ Laguerre unitary ensemble.  In the latter, the transformation $ a \mapsto a+1 $ 
implies a Christoffel-Uvarov transformation of the weight $ w(x) \mapsto (x+t)w(x) $. 
From the work of Uvarov \cite{Uv_1969} we deduce that the orthogonal polynomials $ p_N(x;t;a) $ 
(we adopt the conventions and notations of Section 2 
in \cite{FW_2007}, which should not be confused with their subsequent use in Section \ref{S:H2S}) transform
\begin{equation}
  \hat{p}_N := p_N(x;t;a+1) = \frac{A^{(1,0)}_N}{x+t}\left[ p_{N+1}(x;t;a)p_{N}(-t;t;a)-p_{N}(x;t;a)p_{N+1}(-t;t;a) \right] ,
\end{equation}
where $ A^{(1,0)}_N $ is a normalisation. In the notations of \cite{FW_2007} the three term recurrence coefficients 
transform as
\begin{align}
  \hat{a}^2_{N} & = a^2_{N}\frac{\gamma_{N}^2}{\gamma_{N+1}\gamma_{N-1}}\frac{p_{N+1}(-t;t;a)p_{N-1}(-t;t;a)}{p_{N}(-t;t;a)^2} ,
\\
  \hat{b}_{N} & = b_{N}+a_{N}\frac{p_{N-1}(-t;t;a)}{p_{N}(-t;t;a)}-a_{N+1}\frac{p_{N}(-t;t;a)}{p_{N+1}(-t;t;a)} .
\end{align}
Employing the variables $ Q_N, R_N $ (see the definitions Eqs. (3.41) and (3.52) of \cite{FW_2007}) 
instead of $ p_N, p_{N-1} $ we find that the transformation gives
\begin{align}
  \hat{Q}_N(x) & = \frac{t}{x+t}\left[ Q_N(x)-\frac{Q_N(-t)}{R_{N+1}(-t)}R_{N+1}(x) \right] ,
\\
  \hat{R}_N(x) & = \frac{t}{x+t}Q_N(-t)\left[ \frac{R_{N}(x)}{R_{N}(-t)}-\frac{R_{N+1}(x)}{R_{N+1}(-t)} \right] .
\end{align}
However the second of these equations will suffer a severe cancellation under the hard edge scaling limit
$ t \to s/4N, x \to -z/4N $ as $ N \to \infty $
so we need to be able to handle the subtle cancellations occurring. For this we employ a restatement of the
identity Eq. (3.42) of \cite{FW_2007}
\begin{equation}
  x \theta_N Q_N+(\kappa_N-t)R_N-(\kappa_{N+1}+t)R_{N+1} = 0 ,
\end{equation}  
which gives us an exact relation between $ R_N $ and $ R_{N+1} $. 
We now compute
\begin{equation}
 \hat{R}_N(x) = -\frac{t}{x+t}\frac{\theta_NQ_N(-t)}{R_N(-t)}\frac{tQ_N(-t)R_N(x)+xR_N(-t)Q_N(x)}{(\kappa_N-t)R_N(-t)-t\theta_NQ_N(-t)} .
\end{equation}
We are now in a position to take the hard edge scaling limits, as given in \cite{FW_2007} by
Eq. (5.12) for $ \kappa_N $, Eq. (5.10) for $ \theta_N $, Eq. (5.28) for $ Q_N $ and Eq. (5.29) for $ R_N $.
In addition we employ the identity, \cite[Eq. (5.45)]{FW_2007}
\begin{equation}        
    \frac{v(s;s)}{u(s;s)} = \frac{\xi}{C+a} = -\frac{sC}{\mu+s} .                               
\end{equation}
The final result is (\ref{HE_uST},\ref{HE_vST}) where all dependencies other those other than $ a $
are suppressed.
\end{proof}

\subsection{Special Case $ a\in \Z $}\label{SS:integer_a}
In Subsection 5.2 of \cite{FW_2007} determinantal evaluations of the Painlev\'e variables $ \nu$, $\mu$, $C$ and $\xi $, of the 
isomonodromic components $ u,v $ and $ A_a $ for $ a \in \Z_{\geq 0} $ were given. These were of 
Toeplitz or bordered Toeplitz form and of sizes $ a\times a, (a+1)\times(a+1) $ and $ (a+2)\times(a+2) $
respectively. Here we content ourselves with displaying the first two cases only, which can serve as
initial conditions for the recurrences in Propositions \ref{recurHam} and \ref{recurLax}. In order to signify 
the $ a $-value we append a subscript to the variables. In all that follows $ I_{\sigma}(z) $ refers to the
standard modified Bessel function with index $ \sigma $ and argument $ z $, see \textsection 10.25 of \cite{DLMF}.

\subsubsection{$ a = 0 $}
Some details of the first case $ a=0 $ were given in Propositions 5.9, 5.10 and 5.11 of \cite{FW_2007} and 
we augment that collection by computing the remaining variables. Thus we find for the primary variables
\begin{equation}
 \nu_0(s) = 0, \quad \mu_0(s) = -s, \quad C_0(s) = 0, \quad \xi_0(s) = 0 ,
\end{equation}
for the canonical Hamiltonian variables
\begin{equation}
 p_0(s/4) = 0, \quad
 q_0(s/4) = \frac{\ssqrt{s}}{2}\frac{I_{3}(\ssqrt{s})}{I_{2}(\ssqrt{s})} ,
\label{H_a=0}
\end{equation}
the isomonodromic components
\begin{equation}
  u_0(z;s) = \frac{8}{z}I_{2}(\ssqrt{z}), \quad
  v_0(z;s) = \frac{4}{\ssqrt{z}} I_{3}(\ssqrt{z}) ,
\label{uv_a=0}
\end{equation}
and the distribution of the spacing
\begin{equation}
  A_0(z)=\frac{1}{4}e^{-z/4}\left[I_{2}(\ssqrt{z})^2-I_{1}(\ssqrt{z})I_{3}(\ssqrt{z})\right] .
\end{equation} 
This formula is essentially the same as the gap probability at the hard edge for $ a=2 $, as one can see 
from the $ \mu=0 $ specialisation of Eq. (8.97) in \cite{rmt_Fo}. 
Interestingly we should point out that the moments of the above distribution can be exactly evaluated
and we illustrate this observation by giving the first few examples ($ m_0=1 $)
\begin{align*}
 m_1 & = 4e^2 \left[ I_0(2)-I_1(2) \right] ,
\\
 m_2 & = 32e^2 I_0(2) ,
\\
 m_3 & = 384e^2 \left[ 2I_0(2)+I_1(2) \right] ,
\\
 m_4 & = 2048e^2 \left[ 13I_0(2)+9I_1(2) \right] ,
\\
 m_5 & = 20480e^2 \left[ 55I_0(2)+42I_1(2) \right] ,
\\
 m_6 & = 98304e^2 \left[ 557I_0(2)+441I_1(2) \right] .
\end{align*}

\subsubsection{$ a = 1 $}
This case was not considered in \cite{FW_2007}. We have computed these from the results for the finite rank
deformed Laguerre ensemble, as given in Section~4 of \cite{FW_2007}, and then applied the hard edge scaling
limits given by the Hilb type asymptotic formula Eq. (5.2) therein and the limits of Proposition 5.1 and Corollary 5.2 
of \cite{FW_2007}. For the primary variables we find
\begin{align}
  \nu _1(s)& = \frac{\ssqrt{s}}{2}\frac{I_{3}(\ssqrt{s})}{I_{2}(\ssqrt{s})},
\\
  \mu _1(s)& = -4\ssqrt{s}\frac{I_{3}(\ssqrt{s})}{I_{2}(\ssqrt{s})}-s\frac{I_{3}(\ssqrt{s})^2}{I_{2}(\ssqrt{s})^2},
\\
  C_1(s)& = -3+\frac{\ssqrt{s} I_{2}(\ssqrt{s})}{2 I_{3}(\ssqrt{s})}-\frac{\ssqrt{s} I_{3}(\ssqrt{s})}{2 I_{2}(\ssqrt{s})},
\\
  \xi _1(s)& = \frac{s}{4}-\frac{s I_{2}(\ssqrt{s})^2}{4 I_{3}(\ssqrt{s})^2}+\frac{3 \ssqrt{s} I_{2}(\ssqrt{s})}{2 I_{3}(\ssqrt{s})} ,
\end{align}
the PIII$^{\prime}$ canonical Hamiltonian variables 
\begin{align}
 p_1(s/4) & = 1-\frac{I_{3}(\ssqrt{s})I_{1}(\ssqrt{s})}{I_{2}(\ssqrt{s})^2},
\\
 q_1(s/4) & = \frac{I_{2}(\ssqrt{s})}{2I_{3}(\ssqrt{s})}
        \frac{\ssqrt{s} I_{2}(\ssqrt{s})^2-6 I_{2}(\ssqrt{s}) I_{3}(\ssqrt{s})-\ssqrt{s} I_{3}(\ssqrt{s})^2}
             {I_{1}(\ssqrt{s}) I_{3}(\ssqrt{s})-I_{2}(\ssqrt{s})^2},
\end{align}
the isomonodromic components for generic argument $ s>z>0 $ 
\begin{align}
  u_1(z;s)& = \frac{8\ssqrt{s}}{z I_{3}(\ssqrt{s})}\frac{\ssqrt{s}I_{1}(\ssqrt{s}) I_{2}(\ssqrt{z})- \ssqrt{z}I_{1}(\ssqrt{z}) I_{2}(\ssqrt{s})}{s-z} ,
\\
  v_1(z;s)& = \frac{4\ssqrt{s}I_{2}(\ssqrt{s})}{\ssqrt{z}I_{3}(\ssqrt{s})^2}\frac{\ssqrt{s}I_{2}(\ssqrt{s}) I_{3}(\ssqrt{z})- \ssqrt{z}I_{2}(\ssqrt{z}) I_{3}(\ssqrt{s})}{s-z} ,
\end{align}
and the isomonodromic components on $ s=z $
\begin{align}
  u_1(s;s)& = -\frac{4 I_{1}(\ssqrt{s})}{\ssqrt{s}}+\frac{4 I_{2}(\ssqrt{s})^2}{\ssqrt{s} I_{3}(\ssqrt{s})},
\\
  v_1(s;s)& = -2 I_{2}(\ssqrt{s}) \frac{-s I_{1}(\ssqrt{s})^2+2 \ssqrt{s} I_{1}(\ssqrt{s}) I_{2}(\ssqrt{s})+(8+s) I_{2}(\ssqrt{s})^2}
                                      {\left[\ssqrt{s} I_{1}(\ssqrt{s})-4 I_{2}(\ssqrt{s})\right]^2} ,
\end{align}
and the distribution of the eigenvalue gap is
\begin{align}
  A_1(z) & = 2^{-4}\left[ I_{0}(\ssqrt{z})I_{2}(\ssqrt{z})-I_{1}(\ssqrt{z})^2 \right] \int^{\infty}_{z} {\rm d}s\; e^{-s/4}I_{2}(\ssqrt{s}) \notag 
\\ 
& \quad  + 2^{-3}z^{-1/2}I_{2}(\ssqrt{z}) \int^{\infty}_{z} {\rm d} s\; \ssqrt{s}e^{-s/4}\frac{\left[ \ssqrt{s}I_{1}(\ssqrt{z})I_{2}(\ssqrt{s})-\ssqrt{z}I_{1}(\ssqrt{s})I_{2}(\ssqrt{z}) \right]}{s-z} .
\end{align}
From the point of view of checking one can verify that the above solutions satisfy their respective 
characterising equations.

\subsection{Lax Pairs}\label{SS:HE_LaxPairs}

We now examine the isomonodromic system from the viewpoint of its characterisation as the solution to the 
partial differential systems with respect to $ z $ and $ s $.  
\begin{proposition}[{\cite[Eqs. (5.51,5.52,5.54-7)]{FW_2007}}]
The matrix form of the spectral derivatives \eqref{HE_S:a} and \eqref{HE_S:b} and deformation 
derivatives \eqref{HE_D:a} and \eqref{HE_D:b} yield the Lax pair
\begin{equation}
  \partial_z \Psi = 
  \left\{  \begin{pmatrix} 0 & 0 \\ \frac{1}{4} & 0 \end{pmatrix} 
          +\begin{pmatrix} C & \frac{\displaystyle\mu+s}{\displaystyle s} \\ \xi & -C-a \end{pmatrix}\frac{1}{z-s}
          +\begin{pmatrix} 0 & -\frac{\displaystyle\mu}{\displaystyle s} \\ 0 & -2 \end{pmatrix}\frac{1}{z}
  \right\} \Psi ,
\label{HE_LP:a}
\end{equation}
and
\begin{equation}
  \partial_s \Psi = 
  \left\{  \frac{1}{s}\begin{pmatrix} -C & 0 \\ -\xi & -C \end{pmatrix}
          -\begin{pmatrix} C & \frac{\displaystyle\mu+s}{\displaystyle s} \\ \xi & -C-a \end{pmatrix}\frac{1}{z-s}
  \right\} \Psi .
\label{HE_LP:b}
\end{equation}
This system is essentially equivalent to the isomonodromic system
of the fifth Painlev\'e equation but is the degenerate case.
The system has two regular singularities at $ z=0,s $ and an irregular one at $ z=\infty $
with a Poincar\'e index of $ \tfrac{1}{2} $.
\end{proposition}
 
The form of the isomonodromic system \eqref{HE_S:b} -\eqref{HE_xidefn} is not suitable for computing the
hard-to-soft edge scaling limit, so we need to perform some preliminary transformations on it.

\begin{proposition}   
Under the gauge transformation
\begin{equation}
  u, v \mapsto z^{-1}s^{a/2}(s-z)^{-a/2} u, v
\label{gaugeFxm}
\end{equation}
the spectral derivatives \eqref{HE_S:a} and\eqref{HE_S:b} become
\begin{align}
  z(s-z)\partial_{z}u & = \left[s-z-z(C+\tfrac{1}{2}a)\right] u-(\mu+z) v,
\label{HE_S:c}\\
  z(s-z)\partial_{z}v & = -z\left[\xi +\tfrac{1}{4}(z-s)\right] u+\left[z-s+(C+\tfrac{1}{2}a)z\right] v,
\label{HE_S:d}
\end{align}
whilst the deformation derivatives \eqref{HE_D:a} and \eqref{HE_D:b} become
\begin{align}
  (s-z)s\partial_su & =  (C+\tfrac{1}{2}a)z u+(\mu+s) v ,
\label{HE_D:c} \\
  (s-z)s\partial_sv & = z\xi u+(C+\tfrac{1}{2}a)(z-2s)v .
\label{HE_D:d}
\end{align}
Furthermore let us scale the spectral variable $ z \to sr $.
Consequently Equations~\eqref{HE_S:c} and \eqref{HE_S:d} become
\begin{equation}
  \partial_r \Psi = 
  \left\{  \begin{pmatrix} 0 & 0 \\ \frac{1}{4}s & 0 \end{pmatrix} 
          +\begin{pmatrix} C+a/2 & \frac{\displaystyle\mu+s}{\displaystyle s} \\ \xi & -C-a/2 \end{pmatrix}\frac{1}{r-1}
          +\begin{pmatrix} 1 & -\frac{\displaystyle\mu}{\displaystyle s} \\ 0 & -1 \end{pmatrix}\frac{1}{r}
  \right\} \Psi ,
\label{HE_LP:c}
\end{equation}
and Equations~\eqref{HE_D:c} and \eqref{HE_D:d} become
\begin{equation}
  s\partial_s \Psi = 
  \left\{  \begin{pmatrix} -C-a/2 & 0 \\ -\xi & -C-a/2 \end{pmatrix}
          -\begin{pmatrix} C+a/2 & \frac{\displaystyle\mu+s}{\displaystyle s} \\ \xi & -C-a/2 \end{pmatrix}\frac{1}{r-1}
  \right\} \Psi .
\label{HE_LP:d}
\end{equation}
This system has two regular singularities $ r=0,1 $ and an irregular one at $ r=\infty $ with Poincar\'e 
rank of $ \tfrac{1}{2} $ (due to the nilpotent character of the leading matrix in \ref{HE_LP:c}), 
and is denoted by the symbol $ (1)^2(\tfrac{3}{2}) $.
\end{proposition}

The precise solutions we seek are defined by their local expansions about $ r=0,1 $, and in particular 
the former case. From the general theory of linear ordinary differential equations \cite{Sibuya_1990}, 
\cite{CL_1955} we can deduce the existence of convergent expansions about $ r=0 $ ($ s\neq 0 $)
\begin{align}
 u(r;s) & = \sum_{m=0}^{\infty}\text{u}_{m}(s)r^{\chi_0+m} ,
\label{uExpr=0}\\
 v(r;s) & = \sum_{m=0}^{\infty}\text{v}_{m}(s)r^{\chi_0+m} ,
\label{vExpr=0}
\end{align}
with a radius of convergence of at most unity. The indicial values $ \chi_0 $ are fixed by
\begin{equation}
 \mu\text{v}_0+s\text{u}_0(\chi_0-1) = 0,\qquad \text{v}_0(\chi_0+1) = 0 ,
\end{equation}
and the appropriate solution has $ \text{v}_0=0, \text{u}_0 \neq 0 $ and $ \chi_0=1 $
(actually from \cite[Eqs. (5.38, 5.39)]{FW_2007} we know $ \text{u}_0=s, \text{v}_0=0 $).
The general coefficients are given by the recurrence relations
\begin{multline}
  4m(m+2)s\text{u}_{m} = -[-4(m^2+m-2)s+2(m+2)s(2C+a)+\mu(s-4\xi)]\text{u}_{m-1}
\\ 
     +s\mu\text{u}_{m-2}-2[ 2(m+2)s+(2C+a)\mu+2(m+1)\mu ]\text{v}_{m-1} ,
\label{uCff}
\end{multline}
\begin{equation}
  4(m+2)\text{v}_{m} = (s-4\xi)\text{u}_{m-1}-s\text{u}_{m-2}+2[ 2C+a+2(m+1) ]\text{v}_{m-1} .
\label{vCff}
\end{equation}
The first few terms are given by
\begin{gather}
 \text{u}_{0} = s, \quad \text{v}_{0} = 0, 
\label{uvCff:0}\\
 \text{u}_{1} = -\tfrac{1}{2}(2C+a)s+\tfrac{1}{3}\mu(\xi-\tfrac{1}{4}s), \quad \text{v}_{1} = -\tfrac{1}{3}s(\xi-\tfrac{1}{4}s) .
\label{uvCff:1}
\end{gather}
Similar considerations apply to the local expansions about $ r=1 $ however in the hard-to-soft
edge limit this singularity will diverge to $ \infty $ and we will not be able to draw any simple
conclusions in this case.

\section{Hard to Soft Edge Scaling}\label{S:H2S}
\setcounter{equation}{0}

The hard edge to soft edge scaling limit \cite{BF_2003} will be interpreted as the degeneration of
PIII' to PII. Therefore we begin with a summary of the relevant Okamoto theory for PII.

\subsection{Okamoto PII theory}
Henceforth the canonical variables of the Hamiltonian system for PII will be denoted by $ \{q,p;t,H \} $ 
and should not be confused with the use of the same symbols for PIII'. Conforming to common usage we have 
the parameter relations $ \alpha = \alpha_1 - \half = \half - \alpha_0 $. The PII Hamiltonian is
\begin{equation}
  H = -\half (2q^2 - p + t)p - \alpha_1 q \ ,
\label{Ham-PII}
\end{equation}
and therefore the PII Hamilton equations of motion ($ \dot{\phantom{q}} \equiv d/dt $) are
\begin{equation}
  \dot{q} = p - q^2 - \half t\ , \qquad
  \dot{p} = 2qp + \alpha_1 \ .
\label{PII-Heqn}
\end{equation}
The transcendent $ q(t;\alpha) $ then satisfies the standard form of the second Painlev\'e equation
\begin{equation}
 \ddot{q} = 2q^3+tq+\alpha .
\end{equation} 
The PII Hamiltonian $ H(t) $ satisfies the second-order second-degree differential 
equation of Jimbo-Miwa-Okamoto $\sigma$ form for PII,
\begin{equation} 
  \left( \ddot{H} \right)^2 + 4\left( \dot{H} \right)^3 + 2\dot{H}[t\dot{H}-H] - \quarter\alpha^2_1 = 0 \ .
\label{PII-Hode}
\end{equation}
Using the first two derivatives of the non-autonomous Hamiltonian $H$
\begin{equation}
\begin{split}
  \dot{H}  & = -\half p  \ ,\\
  \ddot{H} & = -qp-\half\alpha_1 \ ,
\end{split}
\label{PII-derH}
\end{equation}
we can recover the canonical variables of \eqref{PII-Heqn}.

\subsection{Degeneration from PIII' to PII}
We know from \cite{BF_2003} that upon the scaling \eqref{eq:1.31a} of the variables and taking 
$a \rightarrow \infty$ the hard edge kernel \eqref{eq:1.31} limits to the soft edge kernel \eqref{eq:2fa}
and furthermore the joint distribution $p_{(2)}^{{\rm hard,}\, a}(x_1,x_2)$ limits to 
$p_{(2)}^{{\rm soft}}(x_1,x_2)$. The same holds true for the relationship between $p_{(1)}^{{\rm hard,}\, a}(s)$
and $p_{(1)}^{{\rm soft}}(s)$. This latter fact helps in our computation of $p_{(2)}^{{\rm soft}}(x_1,x_2)$,
since the evaluation of $p_{(1)}^{{\rm soft}}(s)$ in terms of PII is known from previous work \cite{rmt_Fo}, \cite{FW_2001a} allowing
the limiting form of $\nu(t)$ in \eqref{eq:x1} to be deduced. But $\nu(t)$ is the very same PIII' quantity 
appearing in the evaluation  \eqref{eq:1.16} of  $p_{(2)}^{{\rm hard,}\, a}(s-z,s)$.
 
\begin{proposition} \label{prop:8new}
Let
\begin{align}
  s = a^2 [ 1-2^{2/3}a^{-2/3}\tau ].   
\label{H2S_s}
\end{align}
We have that for $a \rightarrow \infty$
\begin{align}
\nu(s) - \frac{s}{4} + a \rightarrow - \left( \frac{a}{2}\right) ^{2/3} \sigma_{II}(\tau),	 
\label{eq:x2} 
\end{align}
where $ t=-2^{1/3}\tau $,
\begin{align}
\sigma_{II}(\tau) = -2^{1/3} \left. H(t) \right|_{\alpha_1 = 2}	 ,
\label{eq:x2a}
\end{align}
and furthermore
\begin{align}
\sigma_{II}(\tau) \mathop{\sim}_{\tau \rightarrow \infty} \frac{{\rm d}}{{\rm d}\tau} \log K^{{\rm soft}}(\tau,\tau).	  
\label{eq:x3} 
\end{align} 
\end{proposition}

\begin{proof}
We know from \cite[Eq.~(8.84)]{rmt_Fo} that
\begin{align}
 p_{(1)}^{{\rm soft}}(s) = \rho_{(1)}^{{\rm soft}}(s)
	 \exp \left( - \int_s^{\infty} \left( \sigma_{II}(t) - \frac{\mbox{d}}{\mbox{d}t} \log \rho_{(1)}^{{\rm soft}}(t) \right) 
			{\mbox d}t \right) , 
\label{eq:xo}
\end{align}
where $\rho_{(1)}^{{\rm soft}}(s) = K^{{\rm soft}}(s,s)$. On the other hand 
\begin{align}
p_{(1)}^{{\rm soft}}(s) = \lim_{a \rightarrow \infty} 2^{2/3} a^{4/3}
        p_{(1)}^{{\rm hard,}\, a} ( a^2 [1-2^{2/3}a^{-2/3}s] ).
\label{eq:x4}
\end{align}
Substituting \eqref{eq:x1} in the RHS and \eqref{eq:xo} in  the LHS of \eqref{eq:x4}, and comparing the 
respective large $s$ forms implies \eqref{eq:x2}. The boundary condition \eqref{eq:x3} is immediate from
\eqref{eq:xo}.
\end{proof}

It will be shown in the Appendix that the solution of the $\sigma$ form of PII \eqref{PII-Hode} with
$\alpha_1 = 2$ as required by \eqref{eq:x2a}, and subject to the boundary condition \eqref{eq:x3}, can be
generated from the well known Hasting-McLeod solution of PII.

The scaled form of $C$ in \eqref{eq:1.16}, as well as the auxiliary quantities $\mu$ \eqref{HE_DEnu} 
and $\xi$ \eqref{HE_xidefn} can now be found as a consequence of \eqref{eq:x2}.

\begin{proposition}  \label{prop:9new}
Let $s$ be related to $\tau$ by \eqref{H2S_s}, and define $ t = -2^{1/3}\tau $ as before. As $a \rightarrow \infty$
\begin{align}
 \mu(s) & \to  -2^{1/3}a^{4/3}p(t),
\label{H2S_mu}\\
 2C(s)+a & \to 2^{2/3}a^{2/3}\left[q(t)+\frac{2}{p(t)}\right] ,
\label{H2S_C}\\
 \xi(s) & \to  \tfrac{1}{4}a^2-2^{-2/3}a^{4/3}\left[ \left(q(t)+\frac{2}{p(t)}\right)^2-\tfrac{1}{2}p(t) \right].
\label{H2S_xi}
\end{align}
\end{proposition}

\begin{proof}
Simple calculations using \eqref{eq:x2} and (\ref{HE_DEnu}), (\ref{HE_Crel}) and (\ref{HE_xidefn}) give (\ref{H2S_mu}),
(\ref{H2S_C}) and (\ref{H2S_xi}) respectively.
\end{proof}

Now we turn to task of deducing the appropriate scaling of the associated linear systems and their limits
in the hard edge to soft edge transition. There are a handful of references treating the problem of how the 
degeneration scheme of the Painlev\'e equations is manifested from the viewpoint of isomonodromic deformations. 
In comparison to the work \cite{Kapaev_2002} our situation is that of the
degenerate PV case with nilpotent matrix $ A_{\infty} $ as given by Eq. (11) in that work and its reduction to
the case of Eq. (13), again with nilpotent matrix $ A_{\infty} $, which corresponds to PII.
In the more complete examination of the coalescence scheme, as given in \cite{OO_2006}, our reduction is the limit
of the degenerate PV (P5-B case) to that P34, and therefore equivalent to PII. However many details we require are
missing or incomplete in \cite{Kapaev_2002} and \cite{OO_2006}, so we give a fuller account of this scaling and
limit for our example.

\begin{lemma}
Let $ \epsilon = 2^{1/6}a^{-1/3} $. The independent spectral variable scales as $ r = \epsilon^2 x $.
Under the hard-to-soft edge scaling limit $ \epsilon \to 0 $ the isomonodromic components scale as 
$ u = {\rm O}(\epsilon^{-4}) $ and $ v = {\rm O}(\epsilon^{-6}) $.  
\end{lemma}
\begin{proof}
Let us denote the leading order scaling of the expansion coefficients given in \eqref{uExpr=0}, \eqref{vExpr=0} by
$ \text{u}_m = {\rm O}(a^{\omega_m}) $ and $ \text{v}_m = {\rm O}(a^{\lambda_m}) $. Employing the leading order
terms of the auxiliary variables \eqref{H2S_mu}, \eqref{H2S_C} and \eqref{H2S_xi} in the recurrence relations for the
coefficients \eqref{uCff} and \eqref{vCff} we deduce that
\begin{gather*}
   \omega_{m} = \max\{ \omega_{m-1}+\tfrac{2}{3}, \omega_{m-2}+\tfrac{4}{3}, \lambda_{m-1} \} ,
\\
   \lambda_{m} = \max\{ \omega_{m-1}+\tfrac{4}{3}, \omega_{m-2}+2, \lambda_{m-1}+\tfrac{2}{3} \} .
\end{gather*}
In fact all terms on the right-hand side balance each other and are satisfied by the single relation
$ \omega_{m} = \omega_{m-1}+\tfrac{2}{3} = \lambda_{m-1} $. The solution to these is 
$ \omega_m = \tfrac{2}{3}m+2 $, $ \lambda_m = \tfrac{2}{3}m+\tfrac{8}{3} $,
given the initial condition $ \omega_1 = \tfrac{8}{3} $. We then deduce that each term in the expansions has leading
order $ \text{u}_m r^{m+1} = {\rm O}(\epsilon^{-4}) $, $ \text{v}_m r^{m+1} = {\rm O}(\epsilon^{-6}) $,
independent of $ m $. Given that the expansions converge uniformly then the whole sums have the stated leading
order expansions.
\end{proof}

\begin{proposition}
Let the isomonodromic components scale as $ u(r;s) = U(x;t) $, $ v(r;s) = \epsilon^{-2} V(x;t) $, as only the relative
leading orders matter.
As $ \epsilon \to 0 $ the spectral derivative scales to one of the Lax pair for the second Painlev\'e equation
$ t,x \in \R $
\begin{multline}
 \partial_x \begin{pmatrix}
             U \\ V
            \end{pmatrix}
\\
 = \left\{ \begin{pmatrix}
            0 & 0 \\ -\frac{1}{2} & 0 
           \end{pmatrix} x
          +\begin{pmatrix}
            -q-\frac{\displaystyle 2}{\displaystyle p} & -1 \\ \frac{1}{2}(t-p)+\left[q+\frac{\displaystyle 2}{\displaystyle p}\right]^2 & q+\frac{\displaystyle 2}{\displaystyle p}
           \end{pmatrix}
           +\begin{pmatrix}
             1 & p \\ 0 & -1 
            \end{pmatrix} \frac{1}{x}
   \right\} \begin{pmatrix}
             U \\ V
            \end{pmatrix} ,
\label{H2S_spectral}
\end{multline}
and the deformation derivative scales to
\begin{equation}
  \partial_t \begin{pmatrix}
             U \\ V
             \end{pmatrix}
 = \left\{ \begin{pmatrix}
            0 & 0 \\ \frac{1}{2} & 0 
           \end{pmatrix} x
          +\begin{pmatrix}
            0 & 1 \\ 0 & -2\left[ q+\frac{\displaystyle 2}{\displaystyle p} \right]
           \end{pmatrix}
   \right\} \begin{pmatrix}
             U \\ V
            \end{pmatrix} .
\label{H2S_deform}
\end{equation}   
\end{proposition}
\begin{proof}
Our starting point is the Lax pair for the degenerate fifth Painlev\'e system given in Equations (\ref{HE_LP:c}) 
and (\ref{HE_LP:d}). Using the expansions (\ref{H2S_C},\ref{H2S_mu},\ref{H2S_xi}) and (\ref{H2S_s})  we deduce 
that the matrix elements appearing in this pair scale as
\begin{gather}
  -C-a/2+\frac{C+a/2}{r-1} \sim -\epsilon^{-2}h_0 ,
\\
  -C-a/2-\frac{C+a/2}{r-1} \sim \tfrac{1}{2}h_0 x,
\\
  \frac{1}{r}+\frac{C+a/2}{r-1} \sim \epsilon^{-2}\left( \frac{1}{x}-\half h_0 \right) - \half\left( xh_0+h_1 \right) ,
\\
  \quarter s-\frac{s}{\mu+s}\left[ (C+a/2)^2-\quarter a^2 \right]\frac{1}{r-1} \sim \epsilon^{-4}\left[ \half(t-x)+ \tfrac{1}{4}h_0^2-\tfrac{1}{2}p \right] ,
\\
   -\frac{\mu+s}{s}\frac{1}{r-1} \sim 1 ,
\\
  -\frac{\mu}{sr}+\frac{\mu+s}{s}\frac{1}{r-1} \sim \frac{p}{x}-1 + \epsilon^2(p-x) ,
\end{gather} 
where the abbreviation is
\begin{equation}
    h_0 = 2\left( q+\frac{2}{p} \right) .
\end{equation}
Using the scaling for $ u,v $ we deduce a meaningful limit as $ \epsilon\to 0 $ given in Equations~(\ref{H2S_spectral}) 
and (\ref{H2S_deform}). One can check that the compatibility of these two equations is ensured by the requirement 
that $ q,p $ satisfy the Hamiltonian equations of motion (\ref{PII-Heqn}).
\end{proof}

\begin{remark}
For general $ \alpha_1 $ the Lax pair of the PII system is 
\begin{multline}
 \partial_x Y
\\
 = \left\{ \begin{pmatrix}
            0 & 0 \\ -\frac{1}{2} & 0 
           \end{pmatrix} x
          +\begin{pmatrix}
            -q-\frac{\displaystyle \alpha_1}{\displaystyle p} & -1 \\
             \frac{1}{2}(t-p)+\left[q+\frac{\displaystyle \alpha_1}{\displaystyle p}\right]^2 & q+\frac{\displaystyle \alpha_1}{\displaystyle p}
           \end{pmatrix}
           +\begin{pmatrix}
             \frac{1}{2}\alpha_1 & p \\ 0 & -\frac{1}{2}\alpha_1 
            \end{pmatrix} \frac{1}{x}
   \right\} Y ,
\label{LPII:a}
\end{multline}
and
\begin{equation}
  \partial_t Y
 = \left\{ \begin{pmatrix}
            0 & 0 \\ \frac{1}{2} & 0 
           \end{pmatrix} x
          +\begin{pmatrix}
            0 & 1 \\ 0 & -2\left[ q+\frac{\displaystyle \alpha_1}{\displaystyle p} \right]
           \end{pmatrix}
   \right\} Y .
\label{LPII:b}
\end{equation}
Equation (\ref{LPII:a}) has the same form as the nilpotent case of Eq. (13) of \cite{Kapaev_2002}, and in addition both
members of the Lax pair (\ref{LPII:a},\ref{LPII:b}) are a variant of the system Eq. (31), given subsequently in \cite{Kapaev_2002}.
It has also been shown in \cite{KH_1999} that the Lax pair of this latter system is related to that of 
Flashka and Newell \cite{FN_1980} via an ``unfolding'' of the spectral variable supplemented by a gauge transformation
(see also 5.0.54,5 on pg. 175 of \cite{FIKN_2006}). In contrast the Lax pair of Jimbo, Miwa and Ueno \cite{JM_1981a} is not
equivalent to any of those mentioned above. In the hard-to-soft edge scaling the regular 
singularity at $ r=0 $ has transformed into the regular singularity at $ x=0 $; the regular singularity at 
$ r=1 $ has merged with the irregular one at $ r=\infty $ yielding an irregular singularity at $ x=\infty $
with its Poincar\'e rank increased by unity, now being $ \tfrac{3}{2} $. The symbol of the new system is
$ (1)(\tfrac{5}{2}) $.
\end{remark}

The solution we seek can be characterised in a precise way though its expansion about the regular
singularity $ x=0 $. 
\begin{lemma}
Let us assume $ |p(t)| > \delta > 0 $ and $ t, q(t), p(t) $ lie in compact subsets of $ \C $.
The isomonodromic components $ U, V $ have a convergent expansion about $ x=0 $, with indicial exponent $ \chi_0=1 $,
whose leading terms are 
\begin{align}
  U(x;t) & = 2x + \frac{p^2(2q^2-p+t)+2qp-4}{3p}x^2
\nonumber\\       & \qquad + \frac{p \left(p^2(2q^2-p+t)+4pq-8\right)(2q^2-p+t)+2(p^2-16q-4pt)}{48 p}x^3 
\nonumber\\       & \qquad\qquad+ {\rm O}(x^4) ,
\\
  V(x;t) & = \frac{p^2(2q^2-p+t)+8pq+8}{3p^2}x^2
\nonumber\\       & \quad + \frac{p\left(p^2(2q^2-p+t)+12pq+24\right)(2q^2-p+t)+2(5p^2+16q-8pt)}{24p^2}x^3
\nonumber\\       & \qquad\quad+ {\rm O}(x^4) .
\end{align}
\end{lemma}
\begin{proof}
The scaling relations (\ref{H2S_s}, \ref{eq:x2}, \ref{H2S_mu}, \ref{H2S_C}, \ref{H2S_xi}) can be applied
term-wise to the expansions (\ref{uExpr=0}, \ref{vExpr=0}) along with the explicit results for the leading
coefficients (\ref{uvCff:0}, \ref{uvCff:1}). Alternatively one can compute the recurrence relations for the
local expansion of the system (\ref{H2S_spectral}) 
\begin{gather}
  U, V = \sum_{m=0}^{\infty} \text{U}_{m},\text{V}_{m} x^{\chi_0+m} .
\label{PII_x=0}
\end{gather}
In such an analysis one finds for the leading relation
$ p\left((\chi_0-1)\text{U}_{0}-p\text{V}_{0}\right)=0 $ and $ 2p^2(\chi_0+1)\text{V}_{0}=0 $. Clearly
for a well-defined solution at $ x=0 $ we must have $ \text{V}_{0}=0 $ and so $ \chi_0=1 $ with 
$ \text{U}_{0}\neq 0 $. The recurrence relations for the coefficients are given by
\begin{multline}
 2m(m+2)p\text{U}_{m} = [ p^2(2q^2-p+t)+(4-2m)qp-4m ]\text{U}_{m-1}
\\           
  +2p(pq-m)\text{V}_{m-1}-p^2\text{U}_{m-2} ,
\label{UVrecur:a}
\end{multline}
\begin{multline}
 2(m+2)p^2\text{V}_{m} = [ p^2(2q^2-p+t)+8qp+8 ]\text{U}_{m-1}
\\
  +2p(pq+2)\text{V}_{m-1}-p^2\text{U}_{m-2} .
\label{UVrecur:b}
\end{multline}
Given $ \text{U}_{0}=2 $ these recurrences generate the unique solution stated above. From these
recurrence relations it is easy to establish that the local expansions (\ref{PII_x=0}) define an entire
function of $ x $.
\end{proof}

We now have all the preliminary results to obtain the sought Painlev\'e II evaluation of 
$p_{(2)}^{{\rm soft}}$, specified originally as the Fredholm minor \eqref{eq:c1}.

\begin{proposition}
Let $p_{(1)}^{{\rm soft}}(t)$ be given by \eqref{eq:xo}. For some constant $C_0$ still to be determined, 
and boundary conditions on $U$ and $V$ still to be determined
\begin{multline}
  p_{(2)}^{{\rm soft}} (t,t-x) = \\
  C_0  p_{(1)}^{{\rm soft}}(t)\; t^{-5/2} \exp \left( -\tfrac{4}{3}t^{3/2} \right)
		\exp \left( \int_{2^{1/3}t}^{\infty} {\rm d}y \left\{ \left( 2q + \frac{4}{p} \right)(-y)
			-\ssqrt{2y} - \frac{5}{2y} \right\}    \right) 
\\ \times
	\left( U\partial_{x}V-V\partial_{x}U \right)( -2^{1/3}x; -2^{1/3}t). 
\label{eq:ss}
\end{multline}
\end{proposition}

\begin{proof}
Applying the gauge transformation (\ref{gaugeFxm}) to \eqref{eq:1.16} and absorbing the pre-factors $ \exp(-s/4) $ 
and $ s^{a} $ into the integral we have
\begin{multline*}
  p_{(2)}^{{\rm hard,}\, a} (s-z,s) = \hat{C}_a(s_0) \exp\left( \int^{s}_{s_0} \frac{{\rm d}w}{w} \left[\nu(w)-\frac{w}{4}+a+2C(w)+a\right] \right)
\\ \times
		\left( u\partial_z v - v\partial_z u \right)(z;s) ,
\end{multline*}
where $ \hat{C}_a(s_0) $ is a normalisation independent of $ s,z $ but dependent on $a$ and the reference point $ s_0 $.
We are now in a position to apply the limiting forms (\ref{eq:x2},\ref{H2S_C}) to this, thus obtaining
\begin{multline}
  p_{(2)}^{{\rm soft}} (t,t-x) 
\\
     = \widetilde{C}_0 \exp \left( \int_{2^{1/3}t}^{t_0} {\rm d}y \left( H + 2q + \frac{4}{p} \right) (-y) \right) 
	  	\left( U\partial_{x}V-V\partial_{x}U \right)( -2^{1/3}x; -2^{1/3}t).  
\label{eq:sn}
\end{multline}

To proceed further, we make use of \eqref{eq:x2a} and \eqref{eq:xo} to note that for suitable $\widetilde{C}_0$,
\begin{align}
\lim_{t_0\rightarrow \infty} \widetilde{C}_0 \exp \left(\int_{2^{1/3}t}^{t_0} {\rm d}y \, H(-y) \right) = p_{(1)}^{{\rm soft}}(t).   
\label{eq:su1}
\end{align}
Furthermore, since 
\begin{align}
   p_{(1)}^{{\rm soft}}(s)  \mathop{\sim}_{s \rightarrow \infty} \rho_{(1)}^{{\rm soft}}(s) ,
\label{eq:sf}
\end{align}
and
\begin{align}
\rho_{(1)}^{{\rm soft}}(s) =K^{{\rm soft}}(s,s) \mathop{\sim}_{s \rightarrow \infty} \frac{1}{8 \pi s} \exp(-\tfrac{4}{3}s^{3/2}) ,
\label{eq:sf1}
\end{align}
we must have
\begin{align}
H(-y) \mathop{\sim}_{y \rightarrow \infty} \ssqrt{2y} + \frac{1}{y}.	 \label{eq:sf2}
\end{align}
The Hamilton equations \eqref{PII-derH} then imply
\begin{align}
p(-y) & \mathop{\sim}_{y \rightarrow \infty} \ssqrt{\frac{2}{y}} - \frac{2}{y^2},	 
\label{eq:sf3} \\
q(-y) & \mathop{\sim}_{y \rightarrow \infty} -\; \ssqrt{\frac{y}{2}} - \frac{3}{4y},      
\label{eq:sf4} 
\end{align}
and thus
\begin{align*}
\left( 2q + \frac{4}{p} \right) (-y) & \mathop{\sim}_{y \rightarrow \infty}  \ssqrt{2y} + \frac{5}{2y}.     
\end{align*}
Consequentially, for suitable $\widetilde{C}_0$, 
\begin{multline}
\lim_{t_0 \rightarrow \infty} \widetilde{C}_0 \exp \left( \int_{2^{1/3}t}^{t_0} {\rm d}y \left( 2q + \frac{4}{p} \right) (-y) \right)
\\
  =  t^{-5/2} \exp \left(-\tfrac{4}{3}t^{3/2} \right)  \exp \left( \int_{2^{1/3}t}^{\infty} {\rm d}y \left\{ \left( 2q + \frac{4}{p} \right)(-y)-\ssqrt{2y}-\frac{5}{2y} \right\} \right) .
\label{eq:su2}
\end{multline}
Substituting \eqref{eq:su1} and \eqref{eq:su2} in \eqref{eq:sn} gives \eqref{eq:ss}.
\end{proof}

It remains to specify $C_0$ in \eqref{eq:ss}, and furthermore to specify the $x,t \rightarrow \infty$ asymptotic form 
of $U$ and $V$. For this we require the fact, which follows from \eqref{eq:c1}, that for $t,x,t-x \rightarrow \infty$,
\begin{align}
p_{(2)}^{{\rm soft}}(t,t-x) \sim \rho_{(1)}^{{\rm soft}}(t) \rho_{(1)}^{{\rm soft}}(t-x).    
\label{eq:p1p}
\end{align}

\begin{proposition} \label{prop:final_edit_prop}
In \eqref{eq:ss}
\begin{align}
C_0 = \frac{1}{4\pi} ,     
\label{eq:ca1}
\end{align}
and furthermore, for $x,t, t-x \rightarrow \infty $ we have
\begin{align}
U(-x;-t) \sim a(x,t) \exp \left( \tfrac{\ssqrt{2}}{3}\left[t^{3/2} - (t-x)^{3/2} \right] \right) , 
\label{eq:ca2} \\
V(-x;-t) \sim b(x,t) \exp \left( \tfrac{\ssqrt{2}}{3}\left[t^{3/2} - (t-x)^{3/2} \right] \right) , 
\label{eq:ca3}
\end{align}
with
\begin{align}
a(x,t) &= \frac{t^{5/4}}{(t-x)^{1/4}} ,
\label{eq:ca4}  \\
b(x,t) &= - \left( \ssqrt{\frac{t}{2}}  - \ssqrt{\frac{t-x}{2}} \right) a(x,t) .
\label{eq:ca5}
\end{align}
\end{proposition}

\begin{proof}
Substituting \eqref{eq:p1p} in the LHS of \eqref{eq:ss} and making use of \eqref{eq:sf} and \eqref{eq:sf1},
it follows from \eqref{eq:ss} that in the asymptotic region in question 
\begin{multline}
\frac{1}{8 \pi(t-x)} \exp \left( -\tfrac{4}{3}(t-x)^{3/2} \right)
\\
  \sim C_0 t^{-5/2}  \exp \left( -\tfrac{4}{3}t^{3/2} \right) 
		\left( U \partial_x V -V \partial_x U \right) (-2^{1/3}x;-2^{1/3}t).
\label{ap1}
\end{multline}
We see immediately from this that \eqref{eq:ca2} and \eqref{eq:ca3} are valid, for $a(x,t)$, $b(x,t)$ algebraic
functions in $x$ and~$t$ satisfying
\begin{multline}
\frac{1}{8 \pi} \frac{ t^{5/2} } {t-x} \sim 
	C_0  \left[ \left. a(2^{1/3}x,2^{1/3}t)  \partial_y b (-y, 2^{1/3}t) \right|_{y=-2^{1/3}x}  \right.
\\ 	\left.
			  - \left. b(2^{1/3}x,2^{1/3}t)  \partial_y a(-y, 2^{1/3}t) \right|_{y=-2^{1/3}x }  \right] .    
\label{eq:ap2}
\end{multline}
On the other hand, we read off from \eqref{H2S_spectral} that for $x,t \rightarrow \infty$
\begin{align*}
\partial_y \left. U(y;-t) \right|_{y=-x} \sim \left( -q(-t) - \frac{2}{p(-t)} \right) U(-x;-t) - V(-x;-t).
\end{align*}
Making use of the leading terms in \eqref{eq:sf3} and \eqref{eq:sf4} as well as \eqref{eq:ca2}, \eqref{eq:ca3} 
it follows that \eqref{eq:ca5} holds true. This latter formula substituted in \eqref{eq:ap2} implies, 
upon choosing $C_0$ according to \eqref{eq:ca1}, that 
\begin{align*}
 a(x,t)^2 = \frac{t^{5/2}}{\ssqrt{t-x}} ,
\end{align*}
and \eqref{eq:ca4} follows upon taking the positive square root.

We can offer a refinement on the above argument by examining in more detail the isomonodromic system 
in the asymptotic regime $ t\to -\infty $. In this regime the leading order of the differential equations 
(\ref{H2S_spectral}) and (\ref{H2S_deform}) become
\begin{equation}
 \partial_x \begin{pmatrix}
             U \\ V
            \end{pmatrix}
 \sim \begin{pmatrix}
            -\sqrt{-\frac{t}{2}} & -1 \\ -\frac{1}{2}x & \sqrt{-\frac{t}{2}}
   \end{pmatrix} \begin{pmatrix}
             U \\ V
                 \end{pmatrix} ,
\label{Asympt_spectral}
\end{equation}
and
\begin{equation}
  \partial_t \begin{pmatrix}
             U \\ V
             \end{pmatrix}
 \sim \begin{pmatrix}
            0 & 1 \\ \frac{1}{2}x & -2\sqrt{-\frac{t}{2}} 
           \end{pmatrix}
           \begin{pmatrix}
             U \\ V
            \end{pmatrix} .
\label{Asympt_deform}
\end{equation}
In retaining only the leading order terms we have also implicitly assumed that $ x\to \infty $ because
in the regime as $ t\to -\infty $ our solution $ p $ is vanishing, which is the location of the only non-zero,
finite singularity. Clearly
\begin{equation}
  (\partial_x+\partial_t) \begin{pmatrix}
             U \\ V
             \end{pmatrix}
 \sim \begin{pmatrix}
            -\sqrt{-\frac{t}{2}} & 0 \\ 0 & -\sqrt{-\frac{t}{2}} 
           \end{pmatrix}
           \begin{pmatrix}
             U \\ V
            \end{pmatrix} ,
\label{}
\end{equation}
so that $ U(x;t),V(x;t) \sim f(t)g_{1,2}(t-x) $. This means that $ \partial_t f = -\sqrt{-\frac{t}{2}}f $
with a solution proportional to $ \exp\frac{\sqrt{2}}{3}(-t)^{3/2} $. For the remaining factor we have
\begin{equation}
  \begin{pmatrix}
             g_{1}' \\ g_{2}'
  \end{pmatrix}
 \sim \begin{pmatrix}
            \sqrt{-\frac{t}{2}} & 1 \\ \frac{1}{2}x & -\sqrt{-\frac{t}{2}} 
      \end{pmatrix}
      \begin{pmatrix}
             g_{1} \\ g_{2}
      \end{pmatrix} ,
\label{}
\end{equation}
or, in terms of the components
\begin{equation}
   g_{1}'' = \tfrac{1}{2}(x-t)g_{1}, \qquad g_{2} = g_{1}'-\sqrt{-\frac{t}{2}}g_{1}.
\end{equation} 
Thus, with $ {\rm Ai} $ and $ {\rm Bi} $ the two linearly independent solutions of the Airy equation
\cite[\textsection 9.2]{DLMF}, we have
\begin{equation}
  U(x;t) \sim \exp\frac{\sqrt{2}}{3}(-t)^{3/2}
  \left\{ \alpha(t){\rm Ai}(2^{-1/3}(x-t))+\beta(t){\rm Bi}(2^{-1/3}(x-t)) \right\} ,
\end{equation} 
and
\begin{multline}
  V(x;t) \sim \exp\frac{\sqrt{2}}{3}(-t)^{3/2}
\\ \times
  \left\{ -\alpha(t)\left[ 2^{-1/3}{\rm Ai}'(2^{-1/3}(x-t))+\sqrt{-\frac{t}{2}}{\rm Ai}(2^{-1/3}(x-t)) \right] \right.
\\
  \left.   -\beta(t)\left[ 2^{-1/3}{\rm Bi}'(2^{-1/3}(x-t))+\sqrt{-\frac{t}{2}}{\rm Bi}(2^{-1/3}(x-t)) \right] \right\} .
\end{multline}
From these we compute
\begin{multline}
 U\partial_x V-V\partial_x U \sim 
\\
   \exp2\frac{\sqrt{2}}{3}(-t)^{3/2}
                \left\{ -\tfrac{1}{2}(x-t)\left[ \alpha{\rm Ai}+\beta{\rm Bi} \right]^2+2^{-2/3}\left[ \alpha{\rm Ai}'+\beta{\rm Bi}' \right]^2 \right\} .
\end{multline} 
Clearly $ \beta=0 $ in order to suppress the dominant terms, and by employing the exponential asymptotics 
of the Airy function through to second order (the leading order cancels exactly) we have
\begin{equation}
  \left( U\partial_x V-V\partial_x U \right)(x;t) \sim \frac{2^{-4/3}\alpha^2}{4\pi(x-t)}
  \exp\left[ 2\frac{\sqrt{2}}{3}(-t)^{3/2}-2\frac{\sqrt{2}}{3}(x-t)^{3/2} \right] .
\end{equation} 
Employing this into the factorisation of $ p_{(2)}^{\rm soft}(t,t-x) $ as $ t, t-x, x \to \infty $ we
deduce
\begin{equation}
   C_{0}\alpha(-2^{1/3}t)^2 = 2^{2/3}t^{5/2} ,
\end{equation} 
and with $ C_{0}=1/4\pi $ we infer $ \alpha(t)=2^{11/12}\pi^{1/2}(-t)^{5/4} $. This then precisely
reproduces the boundary conditions given by (\ref{eq:ca2}, \ref{eq:ca3}) along with (\ref{eq:ca4}, \ref{eq:ca5}).
A by-product of this argument is that the isomonodromic components can be represented in the forms
\begin{gather}
  U(x;t) \sim -t\frac{{\rm Ai}(2^{-1/3}(x-t))}{{\rm Ai}(-2^{-1/3}t)} ,
\\
  V(x;t) \sim t\frac{2^{-1/3}{\rm Ai}'(2^{-1/3}(x-t))+\sqrt{-\frac{t}{2}}{\rm Ai}(2^{-1/3}(x-t))}{{\rm Ai}(-2^{-1/3}t)} ,
\end{gather}
although this is still only valid in the regime $ t, t-x, x \to -\infty $.
\end{proof}

We remark that as written (\ref{eq:ss}) is not well defined for $t \le 0$. In this region
we should use instead (\ref{eq:sn}), with a suitable $t_0$ to make use of (\ref{eq:ss}) for
$t > 0$ so that $\tilde{C}_0$ can be specified.

In \cite{FW_2007} the analogue of (\ref{eq:ss}), Eq.~(2.1), was used to provide high precision
numerics for the spacing between the two smallest eigenvalues at the hard edge.
But to use (\ref{eq:ss}) to compute the spacing distribution (\ref{eq:c2}) presents additional challenges
to obtain control on the accuracy. The essential problem faced in comparison to \cite{FW_2007},
and in also in comparison to the study \cite{PS_2004}
relating to PDEs based on the Hasting-McLeod PII transcendent, is that our
boundary conditions involve algebraic terms, and thus cannot be determined to
arbitrary accuracy. In relation to $p(-y)$ and $q(-y)$ this can perhaps be overcome by
using the theory of the Appendix to map to the Hasting-McLeod solution. But even so,
the problem of extending the accuracy of the algebraic terms $a(x,t)$ and $b(x,t)$ in
(\ref{eq:ca2}) and (\ref{eq:ca3}) would remain. While these points remain under investigation,
the problem of determining numerics for (\ref{eq:c2}) can be tackled by using a
variant of the Fredholm type expansion (\ref{eq:c1}), as we will now proceed to detail.

\section{Numerical Evaluation of Moments}\label{sect:numerical}
\setcounter{equation}{0}

We provide some numerical data for the distribution of the spacing of the two largest eigenvalues
based on the accurate numerical evaluation of operator determinants that is surveyed in \cite{MR2895091}.
The joint probability distribution of the two largest eigenvalues is amenable to this method since it is 
given in terms of a $2\times 2$ operator matrix determinant:
\begin{multline*}
F(x,y) = \int_{-\infty}^x\int_{-\infty}^y p^\text{soft}_{(2)}(\xi,\eta)\,d\xi\,d\eta\\*[1mm]
= \begin{cases}
E_2^\text{soft}(0;(x,\infty)) & (x \leq y),\\*[2mm]
E_2^\text{soft}(0;(y,\infty))  - \left.\dfrac{\partial}{\partial z} \det\left(I-
\begin{pmatrix}
z \,K^\text{soft} & K^\text{soft} \\
z \,K^\text{soft} & K^\text{soft}
\end{pmatrix}_{|L^2(y,x)\oplus L^2(x,\infty)}\right)\right|_{z=1} & (x > y).
\end{cases}
\end{multline*}
Here, the differentiation with respect to $z$ can be accurately computed by the Cauchy integral formula in the
complex domain \cite{MR2754188}. It has been used in \cite{MR2895091} to calculate the correlation between the
two largest eigenvalues to 11 digits accuracy
\[
\rho = 0.50564\,72315\,9.
\]
This number can be calculated without any differentiation of the joint distribution function $F$ since,
according to a  lemma of Hoeffding \cite{Hoeffding}, the covariance is given by 
\[
\text{cov} = \int_{-\infty}^\infty \int_{-\infty}^\infty (F(x,y)- F(x,\infty)F(\infty,y))\,dx\,dy.
\]
In contrast, the spacing distribution
\[
G(s) = \int_0^s A^\text{soft}(\sigma)\,d\sigma = \int_{-\infty}^\infty \frac{\partial }{\partial y} F(x,y)|_{x=y+s}\,dy
\]
requires numerical differentiation in the real domain which causes a loss of a couple of digits. The differentiation
is done by spectral collocation in Chebyshev points of the first kind and $G$ is numerically represented by polynomial 
interpolation in the same type of points; Table~\ref{table:numbers} tabulates the values of $A^\text{soft}(s)$ and 
$G(s)$ for $s=0(0.05)8.95$ to an absolute accuracy of 8 digits based on a polynomial representation of degree $64$ that 
is accurate to about 9 to 10 digits. Fig.~\ref{fig:1} plots the density function as compared to a histogram obtained 
from $10\,000$ draws from a $1000\times 1000$ GUE at the soft edge.

\begin{table}[tbp]
\caption{\small The first four statistical moments of the density function $A^\text{soft}(s)$.}\label{table:moments}
{\small
\begin{tabular}{cccc}
mean & variance & skewness & excess kurtosis \\*[0.2mm]\hline\\*[-3.5mm]
$1.90435\,049$  &  $0.68325\,206$  &  $0.56229\,2$  &  $0.27009$
\end{tabular}}
\end{table}

The moments of the random variable $S$ representing the spacing are obtained from the following derivative-free 
formulae obtained from partial integration:
\[
\E(S^n) = \int_0^\infty s^n \,d G(s) =  n\int_0^\infty s^{n-1} (1-G(s))\,ds\qquad (n=1,2,\ldots).
\]
This way we have obtained the first four statistical moments shown in Table~\ref{table:moments}; estimates
of the approximation errors by calculations to higher accuracy indicate the given digits to be correctly truncated.
The total computing time was 5 hours for the solution and 30 hours for the higher accuracy control calculation.

\begin{figure}[tbp]
\begin{center}
{\includegraphics[width=0.75\textwidth]{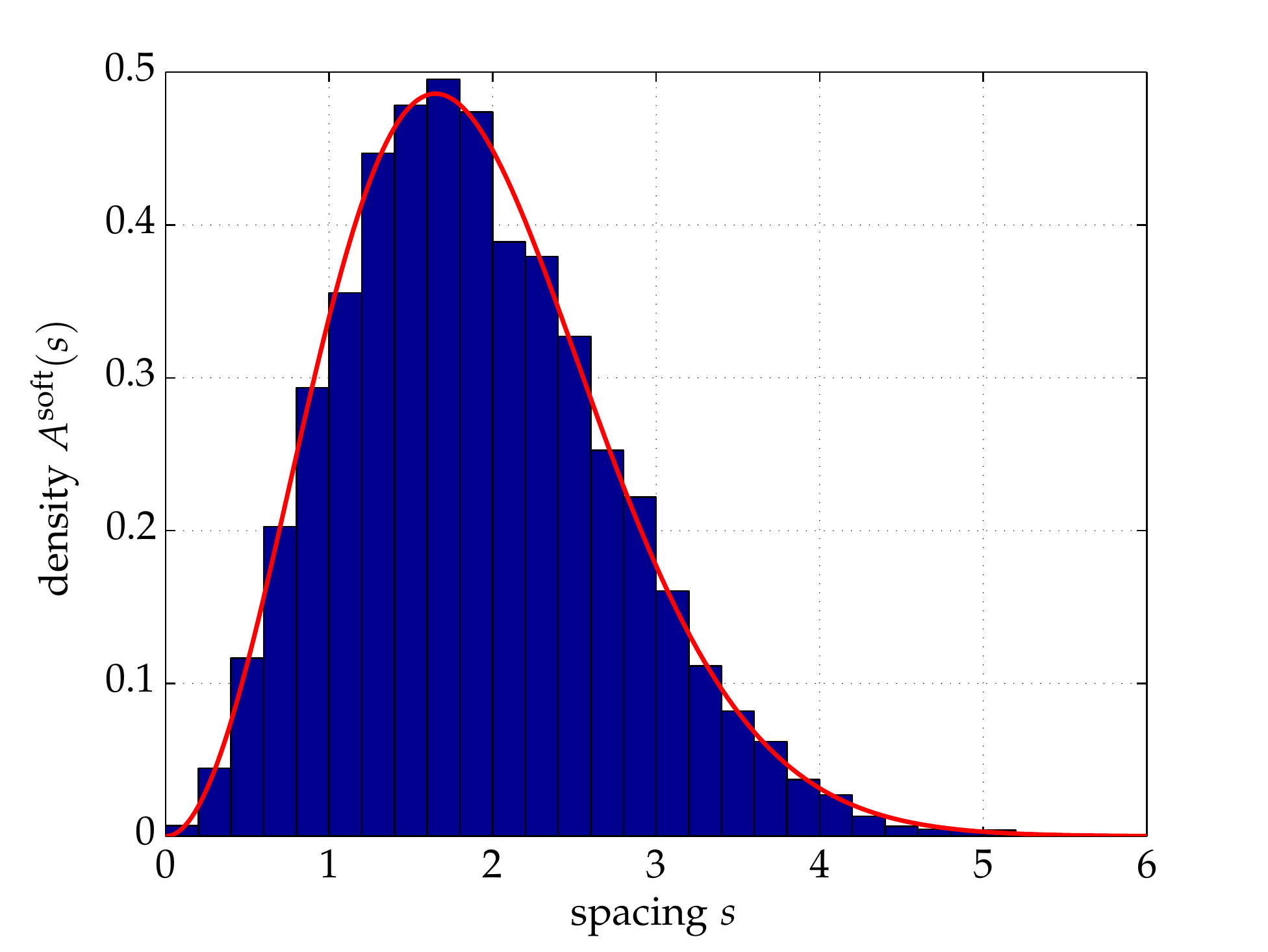}}
\end{center}
\caption{\small A plot of the density function $A^\text{soft}(s)$ as compared to a histogram of $10\,000$ draws 
from a $1000\times 1000$ GUE at the soft edge.} 
\label{fig:1}
\end{figure}

\begin{table}[tbp]
\caption{\small Values of the probability density and distribution function for the spacing $s$ between the two largest eigenvalues; with $s=0(0.05)8.95$.}\label{table:numbers}
{\scriptsize
\begin{minipage}{0.3\textwidth}
\begin{verbatim}
0.00  0.00000000  0.00000000
0.05  0.00124877  0.00002082
0.10  0.00498037  0.00016627
0.15  0.01115087  0.00055952
0.20  0.01968790  0.00132082
0.25  0.03049179  0.00256611
0.30  0.04343713  0.00440570
0.35  0.05837484  0.00694304
0.40  0.07513443  0.01027356
0.45  0.09352660  0.01448370
0.50  0.11334613  0.01965002
0.55  0.13437484  0.02583847
0.60  0.15638476  0.03310386
0.65  0.17914131  0.04148939
0.70  0.20240646  0.05102647
0.75  0.22594192  0.06173454
0.80  0.24951215  0.07362123
0.85  0.27288721  0.08668250
0.90  0.29584551  0.10090300
0.95  0.31817624  0.11625658
1.00  0.33968157  0.13270686
1.05  0.36017860  0.15020793
1.10  0.37950094  0.16870513
1.15  0.39750012  0.18813596
1.20  0.41404648  0.20843092
1.25  0.42902996  0.22951455
1.30  0.44236044  0.25130636
1.35  0.45396785  0.27372186
1.40  0.46380199  0.29667358
1.45  0.47183210  0.32007198
1.50  0.47804618  0.34382651
1.55  0.48245010  0.36784642
1.60  0.48506656  0.39204172
1.65  0.48593385  0.41632392
1.70  0.48510454  0.44060682
1.75  0.48264401  0.46480718
1.80  0.47862894  0.48884531
1.85  0.47314581  0.51264560
1.90  0.46628929  0.53613699
1.95  0.45816068  0.55925333
2.00  0.44886644  0.58193363
2.05  0.43851664  0.60412238
2.10  0.42722362  0.62576958
2.15  0.41510061  0.64683091
2.20  0.40226050  0.66726769
2.25  0.38881472  0.68704686
2.30  0.37487223  0.70614088
2.35  0.36053855  0.72452757
2.40  0.34591504  0.74218991
2.45  0.33109815  0.75911585
2.50  0.31617892  0.77529802
2.55  0.30124247  0.79073346
2.60  0.28636770  0.80542330
2.65  0.27162706  0.81937247
2.70  0.25708634  0.83258934
2.75  0.24280470  0.84508543
2.80  0.22883462  0.85687501
2.85  0.21522206  0.86797485
2.90  0.20200658  0.87840384
\end{verbatim}
\end{minipage}
\hfil
\begin{minipage}{0.3\textwidth}
\begin{verbatim}
2.95  0.18922158  0.88818269
3.00  0.17689453  0.89733363
3.05  0.16504734  0.90588014
3.10  0.15369662  0.91384664
3.15  0.14285408  0.92125827
3.20  0.13252688  0.92814064
3.25  0.12271804  0.93451960
3.30  0.11342681  0.94042107
3.35  0.10464906  0.94587084
3.40  0.09637765  0.95089442
3.45  0.08860281  0.95551689
3.50  0.08131251  0.95976278
3.55  0.07449279  0.96365598
3.60  0.06812806  0.96721964
3.65  0.06220145  0.97047609
3.70  0.05669506  0.97344679
3.75  0.05159022  0.97615229
3.80  0.04686776  0.97861218
3.85  0.04250818  0.98084511
3.90  0.03849187  0.98286872
3.95  0.03479928  0.98469969
4.00  0.03141105  0.98635372
4.05  0.02830817  0.98784555
4.10  0.02547209  0.98918899
4.15  0.02288475  0.99039691
4.20  0.02052876  0.99148132
4.25  0.01838737  0.99245336
4.30  0.01644456  0.99332336
4.35  0.01468507  0.99410087
4.40  0.01309441  0.99479468
4.45  0.01165890  0.99541290
4.50  0.01036561  0.99596294
4.55  0.00920246  0.99645163
4.60  0.00815809  0.99688517
4.65  0.00722193  0.99726924
4.70  0.00638415  0.99760900
4.75  0.00563563  0.99790914
4.80  0.00496793  0.99817391
4.85  0.00437327  0.99840715
4.90  0.00384450  0.99861234
4.95  0.00337504  0.99879259
5.00  0.00295889  0.99895073
5.05  0.00259055  0.99908928
5.10  0.00226503  0.99921050
5.15  0.00197778  0.99931642
5.20  0.00172467  0.99940884
5.25  0.00150198  0.99948939
5.30  0.00130633  0.99955949
5.35  0.00113470  0.99962042
5.40  0.00098434  0.99967332
5.45  0.00085281  0.99971917
5.50  0.00073792  0.99975887
5.55  0.00063769  0.99979321
5.60  0.00055039  0.99982286
5.65  0.00047444  0.99984843
5.70  0.00040846  0.99987047
5.75  0.00035123  0.99988943
5.80  0.00030164  0.99990572
5.85  0.00025873  0.99991970
\end{verbatim}
\end{minipage}
\hfil
\begin{minipage}{0.3\textwidth}
\begin{verbatim}
5.90  0.00022166  0.99993169
5.95  0.00018967  0.99994195
6.00  0.00016210  0.99995073
6.05  0.00013837  0.99995823
6.10  0.00011798  0.99996462
6.15  0.00010047  0.99997007
6.20  0.00008546  0.99997471
6.25  0.00007260  0.99997865
6.30  0.00006161  0.99998200
6.35  0.00005222  0.99998484
6.40  0.00004421  0.99998725
6.45  0.00003739  0.99998928
6.50  0.00003158  0.99999100
6.55  0.00002665  0.99999246
6.60  0.00002246  0.99999368
6.65  0.00001890  0.99999471
6.70  0.00001590  0.99999558
6.75  0.00001335  0.99999631
6.80  0.00001120  0.99999692
6.85  0.00000939  0.99999743
6.90  0.00000786  0.99999786
6.95  0.00000657  0.99999822
7.00  0.00000549  0.99999853
7.05  0.00000458  0.99999878
7.10  0.00000382  0.99999899
7.15  0.00000318  0.99999916
7.20  0.00000264  0.99999931
7.25  0.00000220  0.99999943
7.30  0.00000182  0.99999953
7.35  0.00000151  0.99999961
7.40  0.00000125  0.99999968
7.45  0.00000104  0.99999973
7.50  0.00000086  0.99999978
7.55  0.00000071  0.99999982
7.60  0.00000058  0.99999985
7.65  0.00000048  0.99999988
7.70  0.00000039  0.99999990
7.75  0.00000032  0.99999992
7.80  0.00000027  0.99999993
7.85  0.00000022  0.99999995
7.90  0.00000018  0.99999996
7.95  0.00000015  0.99999996
8.00  0.00000012  0.99999997
8.05  0.00000010  0.99999998
8.10  0.00000008  0.99999998
8.15  0.00000007  0.99999998
8.20  0.00000005  0.99999999
8.25  0.00000004  0.99999999
8.30  0.00000004  0.99999999
8.35  0.00000003  0.99999999
8.40  0.00000002  0.99999999
8.45  0.00000002  1.00000000
8.50  0.00000002  1.00000000
8.55  0.00000001  1.00000000
8.60  0.00000001  1.00000000
8.65  0.00000001  1.00000000
8.70  0.00000001  1.00000000
8.75  0.00000001  1.00000000
8.80  0.00000000  1.00000000
\end{verbatim}
\end{minipage}}
\end{table}


\section*{Acknowledgments}
This research was supported by the Australian Research Council's Centre of Excellence for Mathematics
and Statistics of Complex Systems. The authors would also like
to acknowledge the assistance of Jason Whyte in the preparation of the manuscript.

\appendix

\section*{Appendix}   
\setcounter{section}{1}
\setcounter{equation}{0}

As a technical matter we will need to make use of the Gambier or Folding transformation
for PII. The fundamental domain or Weyl chamber for the PII system can be taken as the interval
$ \alpha \in (-\half,0] $ or $ \alpha \in [0,\half) $, and there exist identities relating the 
transcendents and related quantities at the endpoints of these intervals. In particular, denoting 
the transcendent $ q(t;\alpha) $ and with $ \epsilon^2 = 1 $, $ t = -2^{1/3}s $ we have 
\cite{Gromak_1999}
\begin{equation}
\begin{split}
   -\epsilon\, 2^{1/3} q^2(s;0)
   & = \frac{{\rm d}}{{\rm d}t}q(t;\half\epsilon) - \epsilon\, q^2(t;\half\epsilon) 
	- \half\epsilon\, t \ ,\\
   q(t;\half\epsilon) & =
   \epsilon\, 2^{-1/3} \frac{1}{q(s;0)} \frac{{\rm d}}{{\rm d}s}q(s;0) \ .
\end{split}
\label{PII-ends}
\end{equation}

In addition we will employ the B{\"a}cklund transformation theory of PII as formulated by 
Noumi and Yamada (see \cite{Noumi_2004}) and put to use in the random matrix context by \cite{FW_2001a}.
We define a shift operator corresponding to a translation of the fundamental
weights of the affine Weyl group $ A^{(1)}_1 $,
\begin{equation}
   T_2: \alpha_0 \mapsto \alpha_0-1, \alpha_1 \mapsto \alpha_1+1 \ .
\label{shift-A1}
\end{equation}
The discrete dynamical system generated by the B\"acklund transformations is also integrable and can 
be identified with a discrete Painlev\'e system, discrete dPI. The members of the sequence $ \{q[n]\}_{n=0}^{\infty} $,
generated by the shift operator $ T_2 $ with the parameters $ (\alpha_0-n,\alpha_1+n) $,
are related by a second-order difference equation which is the alternate form of the first discrete 
Painlev\'e equation, a-dPI,
\begin{equation}
   \frac{\alpha + \half + n}{q[n]+q[n+1]} +
   \frac{\alpha - \half + n}{q[n-1]+q[n]} =
   -2q^2[n] - t \ .
\label{adPI-PII}
\end{equation}
The full set of forward and backward difference equations are \cite{Okamoto_1986}
\begin{align}
   q[n-1] & = -q[n] + \frac{\alpha-\half+n}{p[n]-2q[n]^2-t}
   \ ,\label{PII-Bxfm:a} \\
   q[n+1] & = -q[n] - \frac{\alpha+\half+n}{p[n]}
   \ ,\label{PII-Bxfm:b} \\
   p[n-1] & = -p[n]+2q[n]^2+t
   \ ,\label{PII-Bxfm:c} \\
   p[n+1] & = t - p[n] + 2\left( q[n] + \frac{\alpha+\half+n}{p[n]} \right)^2
   \ .\label{PII-Bxfm:d}
\end{align}
In addition one should note that $ H[n+1]=H[n]-q[n+1] $.

\begin{proposition}  \label{prop:notcorollary}
The solution of the second Painlev\'e equation as given by \eqref{PII-Hode} with parameter $ \alpha_1=2 $ 
and boundary condition \eqref{eq:sf2} is generated from the Hastings-McLeod solution by application of the
$ T_2 $ Schlesinger transformation applied twice and the Gambier transformation \eqref{PII-ends}.
\end{proposition}
\begin{proof}    
Firstly we recall that the parameter for the Hastings-McLeod solution is $ \alpha=0, \alpha_1=1/2 $
whereas we have the case of $ \alpha=3/2, \alpha_1=2 $.
Let $ \tau=-2^{-1/3}t $. The leading, and defining, asymptotics of the Hastings-McLeod solution at 
$ \alpha=0, \alpha_1=1/2 $ as $ \tau\to +\infty $ is (for $ \xi=1 $, Eq. (9.47) of \cite{rmt_Fo})
\begin{equation*}
   q(\tau;\alpha_1=1/2) \mathop{\sim}\limits_{\tau \to \infty} {\rm Ai}(\tau) .
\end{equation*} 
Using the inverse Gambier transformation (\ref{PII-ends}) with $ \epsilon=-1 $ we have the solution as
\begin{equation*}
   q(t;\alpha_1=0) \mathop{\sim}\limits_{t \to -\infty} -2^{-1/3}\frac{{\rm Ai'}(-2^{-1/3}t)}{{\rm Ai}(-2^{-1/3}t)} ,
\end{equation*} 
and therefore $ p(t;\alpha_1=0) \sim 0 $ and $ H(t;\alpha_1=0) \sim 0 $ in this regime. Now using the 
Schlesinger transformations (\ref{PII-Bxfm:b},\ref{PII-Bxfm:d}) we deduce
\begin{align*}
  H(t;\alpha_1=2) & = H(t;\alpha_1=1)+q(t;\alpha_1=1)+\frac{1}{p(t;\alpha_1=1)} ,
\\
   & = H(t;\alpha_1=0)+\frac{1}{2[q(t;\alpha_1=0)]^2-p(t;\alpha_1=0)+t} ,
\\
   & \sim \frac{1}{2[q(t;\alpha_1=0)]^2+t} ,
\\
   & \sim 2^{-1/3}\frac{\left[ {\rm Ai}(-2^{-1/3}t) \right]^2}{\left[ {\rm Ai'}(-2^{-1/3}t) \right]^2+2^{-1/3}t\left[ {\rm Ai}(-2^{-1/3}t) \right]^2} ,
\end{align*}
which is asymptotically equivalent to (\ref{eq:sf2}).
\end{proof}


\bibliographystyle{plain}
\bibliography{moment,random_matrices,nonlinear}

\def\cprime{$'$} \def\cprime{$'$} \def\cprime{$'$}
  \def\cydot{\leavevmode\raise.4ex\hbox{.}} \def\cprime{$'$} \def\cprime{$'$}
  \def\cprime{$'$} \def\cprime{$'$} \def\cprime{$'$} \def\cprime{$'$}
  \def\cprime{$'$} \def\cprime{$'$} \def\cprime{$'$} \def\cdprime{$''$}
  \def\cprime{$'$} \def\cprime{$'$} \def\cprime{$'$} \def\cprime{$'$}
  \def\cdprime{$''$} \def\cydot{\leavevmode\raise.4ex\hbox{.}}
  \def\cydot{\leavevmode\raise.4ex\hbox{.}}
  \def\cydot{\leavevmode\raise.4ex\hbox{.}}
  \def\cydot{\leavevmode\raise.4ex\hbox{.}}
  \def\cydot{\leavevmode\raise.4ex\hbox{.}}
  \def\cydot{\leavevmode\raise.4ex\hbox{.}}
  \def\cydot{\leavevmode\raise.4ex\hbox{.}}
  \def\cydot{\leavevmode\raise.4ex\hbox{.}} \def\cprime{$'$} \def\cprime{$'$}
  \def\cprime{$'$} \def\cydot{\leavevmode\raise.4ex\hbox{.}}
  \def\cydot{\leavevmode\raise.4ex\hbox{.}}
  \def\cydot{\leavevmode\raise.4ex\hbox{.}} \def\cprime{$'$}
  \def\cydot{\leavevmode\raise.4ex\hbox{.}} \def\cprime{$'$} \def\cprime{$'$}
  \def\cprime{$'$} \def\cprime{$'$} \def\cprime{$'$} \def\cprime{$'$}
  \def\cprime{$'$} \def\cprime{$'$} \def\cprime{$'$} \def\cprime{$'$}
  \def\cprime{$'$} \def\cprime{$'$} \def\cydot{\leavevmode\raise.4ex\hbox{.}}
  \def\cprime{$'$} \def\cprime{$'$} \def\cprime{$'$} \def\cprime{$'$}
  \def\cprime{$'$} \def\cprime{$'$} \def\cprime{$'$} \def\cprime{$'$}
  \def\cprime{$'$} \def\cprime{$'$} \def\cprime{$'$} \def\cprime{$'$}
  \def\cprime{$'$} \def\cprime{$'$} \def\cprime{$'$} \def\cprime{$'$}
  \def\cprime{$'$} \def\cprime{$'$} \def\cprime{$'$} \def\cprime{$'$}
  \def\cprime{$'$} \def\cprime{$'$} \def\cprime{$'$} \def\cprime{$'$}
  \def\cprime{$'$} \def\cprime{$'$} \def\cydot{\leavevmode\raise.4ex\hbox{.}}
  \def\cprime{$'$} \def\cprime{$'$} \def\cprime{$'$} \def\cprime{$'$}
  \def\cprime{$'$} \def\cprime{$'$} \def\cprime{$'$} \def\cprime{$'$}
  \def\cprime{$'$} \def\cprime{$'$} \def\cprime{$'$} \def\cprime{$'$}
  \def\cprime{$'$} \def\cprime{$'$} \def\cprime{$'$}
  \def\cydot{\leavevmode\raise.4ex\hbox{.}} \def\cprime{$'$} \def\cprime{$'$}
  \def\cprime{$'$} \def\cprime{$'$} \def\cprime{$'$} \def\cprime{$'$}
  \def\cprime{$'$} \def\cprime{$'$} \def\cprime{$'$} \def\cprime{$'$}
  \def\cprime{$'$} \def\cprime{$'$} \def\cprime{$'$} \def\cprime{$'$}
  \def\cprime{$'$} \def\cprime{$'$} \def\cprime{$'$} \def\cprime{$'$}
  \def\cprime{$'$} \def\cprime{$'$} \def\cprime{$'$} \def\cprime{$'$}
  \def\cprime{$'$} \def\cprime{$'$} \def\cprime{$'$} \def\cprime{$'$}
  \def\cprime{$'$} \def\cprime{$'$} \def\cprime{$'$} \def\cprime{$'$}
  \def\cprime{$'$} \def\cprime{$'$}
\begin{thebibliography}{10}

\bibitem{BR_2000}
J.~Baik and E.~M. Rains.
\newblock Limiting distributions for a polynuclear growth model with external
  sources.
\newblock {\em J. Statist. Phys.}, 100(3-4):523--541, 2000.

\bibitem{MR2895091}
F.~Bornemann.
\newblock On the numerical evaluation of distributions in random matrix theory:
  a review.
\newblock {\em Markov Process. Related Fields}, 16:803--866, 2010.

\bibitem{MR2754188}
F.~Bornemann.
\newblock Accuracy and stability of computing high-order derivatives of
  analytic functions by {C}auchy integrals.
\newblock {\em Found. Comput. Math.}, 11:1--63, 2011.

\bibitem{BF_2003}
A.~Borodin and P.~J. Forrester.
\newblock Increasing subsequences and the hard-to-soft edge transition in
  matrix ensembles.
\newblock {\em J. Phys. A}, 36(12):2963--2981, 2003.
\newblock Random matrix theory.

\bibitem{CIK_2010}
T.~Claeys, A.~Its, and I.~Krasovsky.
\newblock Higher-order analogues of the {T}racy-{W}idom distribution and the
  {P}ainlev\'e {II} hierarchy.
\newblock {\em Comm. Pure Appl. Math.}, 63(3):362--412, 2010.

\bibitem{CL_1955}
E.~A. Coddington and N.~Levinson.
\newblock {\em Theory of ordinary differential equations}.
\newblock McGraw-Hill Book Company, Inc., New York-Toronto-London, 1955.

\bibitem{FN_1980}
H.~Flaschka and A.~C. Newell.
\newblock Monodromy- and spectrum-preserving deformations. {I}.
\newblock {\em Comm. Math. Phys.}, 76(1):65--116, 1980.

\bibitem{FIKN_2006}
A.~S. Fokas, A.~R. Its, A.~A. Kapaev, and V.~Yu. Novokshenov.
\newblock {\em Painlev\'e transcendents}, volume 128 of {\em Mathematical
  Surveys and Monographs}.
\newblock American Mathematical Society, Providence, RI, 2006.
\newblock The Riemann-Hilbert approach.

\bibitem{Fo_1993}
P.~J. Forrester.
\newblock The spectrum edge of random matrix ensembles.
\newblock {\em Nucl. Phys. B}, 402:709--728, 1993.

\bibitem{rmt_Fo}
P.~J. Forrester.
\newblock {\em {L}og {G}ases and {R}andom {M}atrices}, volume~34 of {\em London
  Mathematical Society Monograph}.
\newblock Princeton University Press, Princeton NJ, first edition, 2010.

\bibitem{FW_2001a}
P.~J. Forrester and N.~S. Witte.
\newblock Application of the {$\tau$}-function theory of {P}ainlev\'e equations
  to random matrices: {PIV}, {PII} and the {GUE}.
\newblock {\em Comm. Math. Phys.}, 219(2):357--398, 2001.

\bibitem{FW_2002a}
P.~J. Forrester and N.~S. Witte.
\newblock Application of the {$\tau$}-function theory of {P}ainlev\'e equations
  to random matrices: {$\rm P\sb V$}, {$\rm P\sb {III}$}, the {LUE}, {JUE}, and
  {CUE}.
\newblock {\em Comm. Pure Appl. Math.}, 55(6):679--727, 2002.

\bibitem{FW_2007}
P.~J. Forrester and N.~S. Witte.
\newblock The distribution of the first eigenvalue spacing at the hard edge of
  the {L}aguerre unitary ensemble.
\newblock {\em Kyushu J. Math.}, 61(2):457--526, 2007.

\bibitem{Gromak_1999}
V.~I. Gromak.
\newblock {B}{\"a}cklund transformations of {P}ainlev{\'e} equations and their
  applications.
\newblock In R.~Conte, editor, {\em The {P}ainlev{\'e} {P}roperty: {O}ne
  {C}entury later}, CRM Series in Mathematical Physics, pages 687--734.
  Springer Verlag, New York, 1999.

\bibitem{Hoeffding}
W.~Hoeffding.
\newblock {Ma{\ss}stabinvariante Korrelationstheorie}.
\newblock {\em Schr. Math. Inst. u. Inst. Angew. Math. Univ. Berlin},
  5:181--233, 1940.

\bibitem{IIKS_1990}
A.~R. Its, A.~G. Izergin, V.~E. Korepin, and N.~A. Slavnov.
\newblock Differential equations for quantum correlation functions.
\newblock In {\em Proceedings of the Conference on Yang-Baxter Equations,
  Conformal Invariance and Integrability in Statistical Mechanics and Field
  Theory}, volume~4, pages 1003--1037, 1990.

\bibitem{JM_1981a}
M.~Jimbo, T.~Miwa, and K.~Ueno.
\newblock Monodromy preserving deformation of linear ordinary differential
  equations with rational coefficients. {I}. {G}eneral theory and $\tau
  $-function.
\newblock {\em Phys. D}, 2(2):306--352, 1981.

\bibitem{Kapaev_2002}
A.~A. Kapaev.
\newblock Lax pairs for {P}ainlev\'e equations.
\newblock In {\em Isomonodromic deformations and applications in physics
  ({M}ontr\'eal, {QC}, 2000)}, volume~31 of {\em CRM Proc. Lecture Notes},
  pages 37--48. Amer. Math. Soc., Providence, RI, 2002.

\bibitem{KH_1999}
A.~A. Kapaev and E.~Hubert.
\newblock A note on the {L}ax pairs for {P}ainlev\'e equations.
\newblock {\em J. Phys. A}, 32(46):8145--8156, 1999.

\bibitem{Noumi_2004}
M.~Noumi.
\newblock {\em Painlev\'e equations through symmetry}, volume 223 of {\em
  Translations of Mathematical Monographs}.
\newblock American Mathematical Society, Providence, RI, 2004.
\newblock Translated from the 2000 Japanese original by the author.

\bibitem{OO_2006}
Y.~Ohyama and S.~Okumura.
\newblock A coalescent diagram of the {P}ainlev\'e equations from the viewpoint
  of isomonodromic deformations.
\newblock {\em J. Phys. A}, 39(39):12129--12151, 2006.

\bibitem{Okamoto_1986}
K.~Okamoto.
\newblock Studies on the {P}ainlev\'e equations. {I}{I}{I}. {S}econd and fourth
  {P}ainlev\'e equations, ${P}\sb {{\rm {I}{I}}}$ and ${P}\sb {{\rm {I}{V}}}$.
\newblock {\em Math. Ann.}, 275(2):221--255, 1986.

\bibitem{Ok_1987b}
K.~Okamoto.
\newblock Studies on the {P}ainlev\'e equations. {I}{I}. {F}ifth {P}ainlev\'e
  equation ${P}\sb {\rm {v}}$.
\newblock {\em Japan. J. Math. (N.S.)}, 13(1):47--76, 1987.

\bibitem{Ok_1987c}
K.~Okamoto.
\newblock Studies on the {P}ainlev\'e equations. {I}{V}. {T}hird {P}ainlev\'e
  equation ${P}\sb {{\rm {I}{I}{I}}}$.
\newblock {\em Funkcial. Ekvac.}, 30(2-3):305--332, 1987.

\bibitem{DLMF}
F.~W.~J. Olver, D.~W. Lozier, R.~F. Boisvert, and C.~W. Clark.
\newblock {\em NIST Handbook of Mathematical Functions}.
\newblock Cambridge University Press, Cambridge, 2010.

\bibitem{PS_2004}
M.~Pr{\"a}hofer and H.~Spohn.
\newblock Exact scaling functions for one-dimensional stationary {KPZ} growth.
\newblock {\em J. Statist. Phys.}, 115(1-2):255--279, 2004.

\bibitem{S_2012}
G.~Schehr.
\newblock Extremes of {$N$} vicious walkers for large {$N$}: application to the
  directed polymer and {KPZ} interfaces.
\newblock { \tt ar{X}iv:1203.1658v1}, 2012.

\bibitem{Sibuya_1990}
Yasutaka Sibuya.
\newblock {\em Linear differential equations in the complex domain: problems of
  analytic continuation}, volume~82 of {\em Translations of Mathematical
  Monographs}.
\newblock American Mathematical Society, Providence, RI, 1990.
\newblock Translated from the Japanese by the author.

\bibitem{TW_1994a}
C.~A. Tracy and H.~Widom.
\newblock Level-spacing distributions and the {A}iry kernel.
\newblock {\em Comm. Math. Phys.}, 159(1):151--174, 1994.

\bibitem{Uv_1969}
V.~B. Uvarov.
\newblock The connection between systems of polynomials that are orthogonal
  with respect to different distribution functions.
\newblock {\em USSR Comput. Math. and Math. Phys.}, 9:25--36, 1969.

\bibitem{WW_1958}
E.~T. Whittaker and G.~N. Watson.
\newblock {\em A course of modern analysis. {A}n introduction to the general
  theory of infinite processes and of analytic functions: with an account of
  the principal transcendental functions}.
\newblock Fourth edition. Reprinted. Cambridge University Press, New York,
  1958.

\end{thebibliography}
\nopagebreak

\end{document}